\documentclass[11pt]{amsart}
\usepackage{latexsym,amssymb,amsfonts,amsmath,graphicx,fullpage,url,subfigure}
\usepackage[vcentermath,enableskew]{youngtab}
\usepackage[all]{xy}
\usepackage{multirow,bigcircle,verbatim}
\usepackage{tikz}
\usepackage{xspace}
\def\supertiny{ \font\supertinyfont = cmr10 at 2pt \relax \supertinyfont}
\newtheorem{theorem}{Theorem}

\newtheorem{example}[theorem]{Example}
\newtheorem{definition}[theorem]{Definition}

\newtheorem{proposition}[theorem]{Proposition}

\newtheorem{lemma}[theorem]{Lemma}
\newtheorem{corollary}[theorem]{Corollary}

\newtheorem{conjecture}[theorem]{Conjecture}
\numberwithin{equation}{section}
\numberwithin{theorem}{section}

\DeclareMathOperator*{\pro}{Pro}
\DeclareMathOperator*{\row}{Row}
\DeclareMathOperator*{\spro}{SPro}
\DeclareMathOperator*{\aposet}{\mathbf{A}}
\DeclareMathOperator*{\tposet}{\mathbf{T}}
\newcommand{\rcposet}{rc-poset}

\newcounter{x}
\newcounter{y}
\newcounter{z}

\newcommand\xaxis{210}
\newcommand\yaxis{-30}
\newcommand\zaxis{90}

\newcommand\topside[3]{
  \fill[fill=white, draw=black,shift={(\xaxis:#1)},shift={(\yaxis:#2)},
  shift={(\zaxis:#3)}] (0,0) -- (30:1) -- (0,1) --(150:1)--(0,0);
}

\newcommand\leftside[3]{
  \fill[fill=gray, draw=black,shift={(\xaxis:#1)},shift={(\yaxis:#2)},
  shift={(\zaxis:#3)}] (0,0) -- (0,-1) -- (210:1) --(150:1)--(0,0);
}

\newcommand\rightside[3]{
  \fill[fill=black, draw=black,shift={(\xaxis:#1)},shift={(\yaxis:#2)},
  shift={(\zaxis:#3)}] (0,0) -- (30:1) -- (-30:1) --(0,-1)--(0,0);
}

\newcommand\cube[3]{
  \topside{#1}{#2}{#3} \leftside{#1}{#2}{#3} \rightside{#1}{#2}{#3}
}

\newcommand\planepartition[1]{
 \setcounter{x}{-1}
  \foreach \a in {#1} {
    \addtocounter{x}{1}
    \setcounter{y}{-1}
    \foreach \b in \a {
      \addtocounter{y}{1}
      \setcounter{z}{-1}
      \foreach \c in {1,...,\b} {
        \addtocounter{z}{1}
        \cube{\value{x}}{\value{y}}{\value{z}}
      }
    }
  }
}

\begin{document}

\title{Promotion and Rowmotion}

\author[J.\ Striker]{Jessica Striker}
\email{jessica@math.umn.edu}

\author[N.\ Williams]{Nathan Williams}
\email{will3089@math.umn.edu}

\begin{abstract}
We present an equivariant bijection between two actions---promotion and rowmotion---on order ideals in certain posets.  This bijection simultaneously generalizes a result of R.\ Stanley concerning promotion on the linear extensions of two disjoint chains and certain cases of recent work of D.\ Armstrong, C.\ Stump, and H.\ Thomas on noncrossing and nonnesting partitions.  We apply this bijection to several classes of posets, obtaining equivariant bijections to various known objects under rotation.  We extend the same idea to give an equivariant bijection between alternating sign matrices under rowmotion and under B.\ Wieland's gyration.  Finally, we define two actions with related orders on alternating sign matrices and totally symmetric self-complementary plane partitions.
\end{abstract}

\maketitle

\section{Introduction}
\label{sec:introduction}

In this paper, we relate M.-P.\ Sch\"{u}tzenberger's action promotion ($\pro$) and an action that has been rediscovered and renamed several times---it has, at various points, been called $F$~\cite{brouwerschrijver}, $f$~\cite{deza1990loops, fon1993orbits, cameron1995orbits}, $\psi$~\cite{stanley2009promotion}, $\mathfrak{X}$~\cite{panyushev2009orbits}, the Panyushev action and complement~\cite{bessis701792cyclic, Armstrong2011}, and even the Fon-der-Flaass action~\cite{RushShiREU}.  Because we will interpret this action as acting on rows of certain posets, we call it rowmotion ($\row$).

\begin{definition} \rm
\label{def:row}
	Let $\mathcal{P}$ be a poset, and let $I \in J(\mathcal{P})$.  Then $\row(I)$ is the order ideal generated by the minimal elements of $\mathcal{P}$ not in $I$.
\end{definition}

The motivation for relating promotion and rowmotion comes from the following two results.  In his 2009 survey paper ``Promotion and Evacuation''~\cite{stanley2009promotion}, R.\ Stanley gave an equivariant bijection between linear extensions of two disjoint chains $[n] \oplus [k]$ under $\pro$ and order ideals of the product of two chains $[n] \times [k]$ under $\row$.

In 2011, D.\ Armstrong, C.\ Stump, and H.\ Thomas then gave a beautiful uniformly-characterized equivariant bijection between noncrossing partitions under Kreweras complementation and nonnesting partitions under rowmotion~\cite{Armstrong2011}.  Restricting to type $A$---losing both uniformity and the full generality of their result---we may interpret their equivariant bijection as passing from linear extensions of $[2] \times [n]$ under $\pro$ to order ideals of the type $A$ positive root poset $\Phi^+(A_n)$ under $\row$.

We give a new proof of these two equivariant bijections between linear extensions and order ideals by simultaneously generalizing them as a single theorem about \rcposet{}s---certain posets whose elements and covering relations fit into \textbf{r}ows and \textbf{c}olumns.  This theorem gives an equivariant bijection between the order ideals of an \rcposet{} $\mathcal{R}$ under $\pro$ and $\row$ by interpreting promotion as an action on the \emph{columns} of order ideals of $\mathcal{R}$ and rowmotion as an action on the \emph{rows}.

Armed with promotion, we obtain simple equivariant bijections from the order ideals of $[n]\times[k]$, $J([2]\times[n])$, positive root posets of classical type, and $[2]\times[m]\times[n]$ under rowmotion to various known objects under rotation.

Finally,
we apply this theory to alternating sign matrices (ASMs) and totally symmetric self-complementary plane partitions (TSSCPPs). We interpret B.\ Wieland's gyration action on ASMs in terms of rowmotion on the ASM poset.
We also define two actions with related orders on ASMs and TSSCPPs and speculate on the application of these actions to the open problem of finding an explicit bijection between these two sets of objects.

The remainder of the paper is structured as follows.  In Section~\ref{sec:definitions}, we review basic notions about posets, define promotion and rowmotion, and recall the cyclic sieving phenomenon.  We briefly summarize the history of rowmotion in Section~\ref{sec:history}, and build a framework for our results in Sections~\ref{sec:togglegroup} and~\ref{sec:rcposets} by recalling P.\ Cameron and D.\ Fon-der-Flaass's permutation group on order ideals of a poset~\cite{cameron1995orbits}---which we call the toggle group---and by defining \rcposet{}s.  We characterize promotion and rowmotion in terms of the toggle group of an \rcposet{} in Section~\ref{sec:prorowtoggle}, which allows us to give an equivariant bijection between promotion and rowmotion for \rcposet{}s in Section~\ref{sec:conjugate}.  We apply this equivariant bijection in Section~\ref{sec:othertypes} to the product of two chains and the types $A$ and $B$ positive root posets, thereby recovering the corresponding results in~\cite{stanley2009promotion} and~\cite{Armstrong2011}.  In Section~\ref{sec:threedim} we consider the type $D$ positive root poset and plane partitions.  Finally, we apply this perspective to ASMs and TSSCPPs in Section~\ref{sec:asmtsscpp}. 

\section{Definitions}
\label{sec:definitions}

\subsection{Poset Terminology}

Recall that a \emph{poset} $\mathcal{P}$ is a set with a binary relation ``$\leq$'' that is
reflexive, antisymmetric, and transitive.

\begin{definition} \rm
	An \emph{order ideal} of a poset $\mathcal{P}$ is a set $I \subseteq \mathcal{P}$ such that if $p \in I$ and $p' \leq p$, then $p' \in I$.  We write $J(\mathcal{P})$ for the set of all order ideals of $\mathcal{P}$.
\end{definition}

Recall that $J(\mathcal{P})$ forms a distributive lattice under inclusion.

\begin{definition} \rm
	A \emph{partition} $\lambda = (\lambda_1, \lambda_2, \ldots, \lambda_k)$ is a finite sequence of weakly decreasing positive integers.  
\end{definition}
 
Using English notation, we think of the boxes in a Ferrers diagram of a partition as the elements of a poset, where $x < y$ if the box $x$ is weakly to the left and above the box $y$.  For example, in \tiny $\young(abc,de)$ \normalsize, we have $b<c$ and $b<e$, but $c \not < e$ and $e \not < c$.  For $\mu \subseteq \lambda$, a skew Ferrers diagram $\lambda / \mu$ consists of the boxes in the Ferrers diagram of $\lambda$ that are not in $\mu$.

\begin{definition} \rm
	Let $\mathcal{P}$ have $n$ elements and let $[n]=\{1,2,\ldots,n\}$.  A \emph{linear extension} of $\mathcal{P}$ is a bijection $\mathcal{L}: \mathcal{P} \to [n]$ such that if $p < p'$, then $\mathcal{L}(p) < \mathcal{L}(p').$  We call linear extensions of a skew Ferrers diagram \emph{Standard Young Tableaux (SYT)}.  We write $\mathcal{L}(\mathcal{P})$ for the set of all linear extensions of $\mathcal{P}$.
\end{definition}

\subsection{Promotion}
In 1972, M.-P.\ Sch\"{u}tzenberger defined promotion as an action on linear extensions~\cite{schutzenberger1972promotion}.  We will denote promotion by $\pro$.

\begin{definition} \rm
\label{def:pro}
	Let $\mathcal{L}$ be a linear extension of a poset $\mathcal{P}$ and let $\rho_i$ act on $\mathcal{L}$ by switching $i$ and $i+1$ if they are not the labels of two elements with a covering relation.  We define the \emph{promotion} of $\mathcal{L}$ to be $\pro(\mathcal{L}) = \rho_{n-1} \rho_{n-2} \cdots \rho_1 (\mathcal{L})$.
\end{definition}

Note that promotion can also be defined using jeu-de-taquin, though we will not use this equivalent definition here.  Since each step of promotion can be reversed, $\pro$ is a bijection on SYT of a specified shape.

\subsection{Rowmotion}

In 1973, P.\ Duchet defined an action on hypergraphs~\cite{duchet1974hypergraphes}.  This action was generalized by A.\ Brouwer and A.\ Schrijver to an arbitrary poset in~\cite{brouwerschrijver}.  Because we will interpret the action as acting on rows, we will call it \emph{rowmotion}.  We will denote rowmotion by $\row$.


{
\renewcommand{\thetheorem}{\ref{def:row}}
\begin{definition}
	Let $\mathcal{P}$ be a poset, and let $I \in J(\mathcal{P})$.  Then $\row(I)$ is the order ideal generated by the minimal elements of $\mathcal{P}$ not in $I$.
\end{definition}
\addtocounter{theorem}{-1}
}

As explained in~\cite{deza1990loops}, one motivation for this definition was to study the orbits of the data defining a matroid.  For example, working within a Boolean algebra, applying $\row$ to the order ideal generated by the bases of a matroid gives the order ideal generated by the circuits. For more on the history of rowmotion, see Section~\ref{sec:history}.

\subsection{The Cyclic Sieving Phenomenon}
\label{sec:csp}

The Cyclic Sieving Phenomenon was introduced by V.~Reiner, D.\ Stanton, and D.\ White as a generalization of J.\ Stembridge's $q=-1$ phenomenon~\cite{reiner2004cyclic}.

\begin{definition} [V.\ Reiner, D.\ Stanton, D.\ White] \rm
	Let $X$ be a finite set, $X(q)$ a generating function for $X$, and $C_n$ the cyclic group of order $n$ acting on $X$.  Then the triple $(X, X(q), C_n)$ exhibits the \emph{Cyclic Sieving Phenomenon (CSP)} if for $c \in C_n,$ 
	$$X(\omega(c)) = \left| \{ x \in X : c(x) = x \} \right|,$$
where $\omega: C_n \to \mathbb{C}$ is an isomorphism of $C_n$ with the $n$th roots of unity.
\end{definition}

As an example, we have the following theorem.  We use the notation $\binom{[n]}{k}$ for the subsets of $[n]$ of size $k$, and the $q$-analogues $[n]_q = \frac{1-q^n}{1-q}$, $[n]!_q = \prod_{i=1}^n [i]_q$, and $\binom{n}{k}_q = \frac{[n]!_q}{[k]!_q [n-k]!_q}$.

\begin{theorem}[V.\ Reiner, D.\ Stanton, D.\ White]
\label{thm:denni}
	Let $C_n$ act on $\binom{[n]}{k}$ by the cycle $(1,2,\ldots,n)$.  Then $(\binom{[n]}{k},\binom{n}{k}_q,C_n)$ exhibits the CSP.
\end{theorem}

In general, both rowmotion and promotion have orders that are hard to predict.

\begin{example}\rm
	Promotion has order $7,554,844,752$ on SYT of shape $(8,6)$.
\end{example}

M.\ Haiman and D.\ Kim classified those SYT with $n$ boxes on which promotion has order $n$ or $2n$---the generalized staircases, which include rectangles, staircases, and double staircases~\cite{haiman1992characterization}.  In his thesis, B.\ Rhoades proved a cyclic sieving phenomenon for rectangular SYT under promotion, thereby determining the orbit structure~\cite{rhoades2010cyclic}.  Similar results for other shapes are limited.

\begin{theorem} [B.\ Rhoades]
	Let $\lambda=(n,n,\ldots, n)$ be a rectangular partition of $n\cdot m$, and let $SYT(\lambda)$ be the set of SYT of shape $\lambda$.  Let $C_n$ act on $SYT(\lambda)$ by promotion and let $$f^\lambda(q) = \frac{[n \cdot m]!_q}{\prod_{i,j} [h_{i,j}]_q}$$ be the $q$-analogue of the hook-length formula for $SYT(\lambda)$.  Then $(SYT(\lambda), f^\lambda(q), C_{n\cdot m})$ exhibits the CSP.
\end{theorem}

\section{History}
\label{sec:history}

In this section, we recall known results for rowmotion acting on $[n]\times [k]$ and positive root posets.  We phrase these results as equivariant bijections between linear extensions of a poset under promotion and order ideals of a related poset under rowmotion.

\subsection{Products of Two Chains}

The problem of determining the order of rowmotion on the product of two chains was proposed in a 1974 paper by A.\ Brouwer and A.\ Schrijver~\cite{brouwerschrijver}.  After showing that the order of rowmotion on a Boolean algebra failed to adhere to a conjectural pattern for $n>4$, the two proved the following theorem.

\begin{theorem}[A.\ Brouwer, A.\ Schrijver]
	$J([n]\times[k])$ under $\row$ has order $n+k$.
\end{theorem}

In 1992, D.\ Fon-der-Flaass used a clever combinatorial model to refine this result~\cite{fon1993orbits}.

\begin{theorem}[D.\ Fon-der-Flaass]
\label{thm:fon_n_k}
	The length of any orbit of $J([n]\times[k])$ under $\row$ is $(n+k)/d$ for some $d$ dividing both $n$ and $k$.  Any number of this form is the length of some orbit.
\end{theorem}

In his 2009 survey paper~\cite{stanley2009promotion}, R.\ Stanley noted that there was an equivariant bijection between promotion and rowmotion.  This completely resolved the original problem.

\begin{theorem}[R.\ Stanley]
\label{thm:prorowbinomial}
	There is an equivariant bijection between $\mathcal{L}([n] \oplus [k])$ under $\pro$ and $J([n] \times [k])$ under $\row$.
\end{theorem}

Note that $\mathcal{L}([n] \oplus [k])$ can be thought of as skew SYT of shape $(n+k,k) / (k)$.  $\mathcal{L}([n] \oplus [k])$ is also in equivariant bijection with the set $\binom{[n+k]}{k}$ under the cycle $(1,2,\ldots,n+k)$, so that Theorem~\ref{thm:denni} applies.  Figure~\ref{ex:example34} illustrates this theorem for the case $n=k=2$.

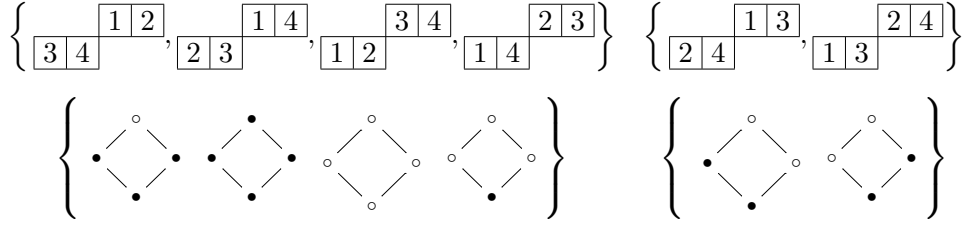
\begin{figure}[ht]
\begin{center}
$\begin{array}{cc}
\left \{ \begin{gathered} $$ \young(::12,34), \young(::14,23), \young(::34,12), \young(::23,14) $$ \end{gathered} \right\} & \left \{ \begin{gathered} $$ \young(::13,24), \young(::24,13)$$\end{gathered} \right\} \vspace{10pt} \\ 
 \left \{ \begin{gathered}\scalebox{0.7}{$$\xymatrix @-1.2pc {
& \circ & \\
\ar@{-}[ur] \bullet & & \ar@{-}[ul] \bullet \\
& \bullet \ar@{-}[ur] \ar@{-}[ul] &} \text{ } \xymatrix @-1.2pc {
& \bullet & \\
\ar@{-}[ur] \bullet & & \ar@{-}[ul] \bullet \\
& \bullet \ar@{-}[ur] \ar@{-}[ul] &} \text{ } \xymatrix @-1.0pc {
& \circ & \\
\ar@{-}[ur] \circ & & \ar@{-}[ul] \circ \\
& \circ \ar@{-}[ur] \ar@{-}[ul] &} \text{ } \xymatrix @-1.2pc {
& \circ & \\
\ar@{-}[ur] \circ & & \ar@{-}[ul] \circ \\
& \bullet \ar@{-}[ur] \ar@{-}[ul] & }$$}\end{gathered} \right \} & \left \{ \begin{gathered}\scalebox{0.7}{$$ \xymatrix @-1.0pc {
& \circ & \\
\ar@{-}[ur] \bullet & & \ar@{-}[ul] \circ \\
& \bullet \ar@{-}[ur] \ar@{-}[ul] &} \text{ } \xymatrix @-1.2pc {
& \circ & \\
\ar@{-}[ur] \circ & & \ar@{-}[ul] \bullet \\
& \bullet \ar@{-}[ur] \ar@{-}[ul] &}$$}\end{gathered} \right \} \\ 
\end{array}$
\end{center}
\caption{The two orbits of $\mathcal{L}([2]\oplus[2])$ under $\pro$, and the two orbits of $J([2] \times [2])$ under $\row$.}
\label{ex:example34}
\end{figure}

\subsection{Positive Root Posets}

Let $W$ be a Weyl group for a root system $\Phi(W)$.

\begin{definition}
	We denote the positive root poset of type $W$ as $\Phi^+(W)$, where if $\alpha,\beta \in \Phi^+ (W)$, then $\alpha \leq \beta$ if $\beta - \alpha$ is a nonnegative sum of positive roots.
\end{definition}

The set of positive roots for the classical types is given below.
\begin{itemize}
\item
 $\Phi^+ (A_n) = \{ e_i - e_j | 1\leq i < j \leq n+1 \}$,
\item
$\Phi^+ (B_n) = \{ e_i \pm e_j | 1 \leq i < j \leq n \} \cup \{ e_i | 1 \leq i \leq n \}$,
\item
$\Phi^+ (C_n) = \{ e_i \pm e_j | 1 \leq i < j \leq n \} \cup \{ 2e_i | 1 \leq
i \leq n \}$, and
\item
$\Phi^+ (D_n) = \{e_i \pm e_{j} | 1 \leq i < j \leq n \}.$
\end{itemize}

Note that $\Phi^+ (C_n)$ is isomorphic to $\Phi^+ (B_n)$.  The positive root posets for $A_3$, $B_3$, $C_3$, and $D_4$ are given in Figure~\ref{ex:rootposets}.

\begin{figure}[ht]
\begin{center}
$\begin{array}{|l|l|}
\hline
\Phi^+ (A_3)
\raisebox{-27pt}
{{\scalebox{0.6}{$$\xymatrix @-1.2pc {
& & e_1-e_4 & & \\
& e_1-e_3 \ar@{-}[ur] & & e_2-e_4 \ar@{-}[ul] & \\
e_1 - e_2 \ar@{-}[ur] & & e_2-e_3 \ar@{-}[ur] \ar@{-}[ul] & & e_3-e_4
\ar@{-}[ul]}$$}}}
&
\Phi^+ (B_3)
\subfigure
{\scalebox{0.7}{$$\xymatrix @-1.2pc {
& & & & e_1+e_2 \\
& & & e_1+e_3 \ar@{-}[ur] & \\
& & e_1 \ar@{-}[ur] & & e_2+e_3 \ar@{-}[ul] \\
& e_1-e_3 \ar@{-}[ur] & & e_2 \ar@{-}[ur] \ar@{-}[ul] & \\
e_1 - e_2 \ar@{-}[ur] & & e_2-e_3 \ar@{-}[ur] \ar@{-}[ul] & & e_3
\ar@{-}[ul]}$$}} \\
\hline
\Phi^+ (C_3)
\subfigure
{\scalebox{0.7}{
$$ \xymatrix @-1.2pc {
& & & & 2e_1 \\
& & & e_1+e_2 \ar@{-}[ur] & \\
& & e_1+e_3 \ar@{-}[ur] & & 2e_2 \ar@{-}[ul] \\
& e_1-e_3 \ar@{-}[ur] & & e_2+e_3 \ar@{-}[ur] \ar@{-}[ul] & \\
e_1 - e_2 \ar@{-}[ur] & & e_2-e_3 \ar@{-}[ur] \ar@{-}[ul] & & 2e_3
\ar@{-}[ul]}$$}} &
\Phi^+ (D_4)
\subfigure
{\scalebox{0.6}{
$$ \xymatrix @-1.2pc {
& & & & & &e_1+e_2\\
& & & & & e_1+e_3 \ar@{-}[ur] &\\
& & e_1-e_4 \ar@{-}[urrr] & & e_1+e_4 \ar@{-}[ur] & &e_2+e_3 \ar@{-}[ul]\\
& e_1-e_3 \ar@{-}[ur] \ar@{-}[urrr] & & e_2-e_4 \ar@{-}[ul] \ar@{-}[urrr] & &
e_2+e_4 \ar@{-}[ul] \ar@{-}[ur]&\\
e_1 - e_2 \ar@{-}[ur] & & e_2-e_3 \ar@{-}[ur] \ar@{-}[ul] \ar@{-}[urrr] & & e_3
- e_4 \ar@{-}[ul] & & e_3+e_4 \ar@{-}[ul]}$$}} \\
\hline
\end{array}$
\end{center}
\caption{The positive root posets $A_3$, $B_3$, $C_3$, and $D_4$.}
\label{ex:rootposets}
\end{figure}
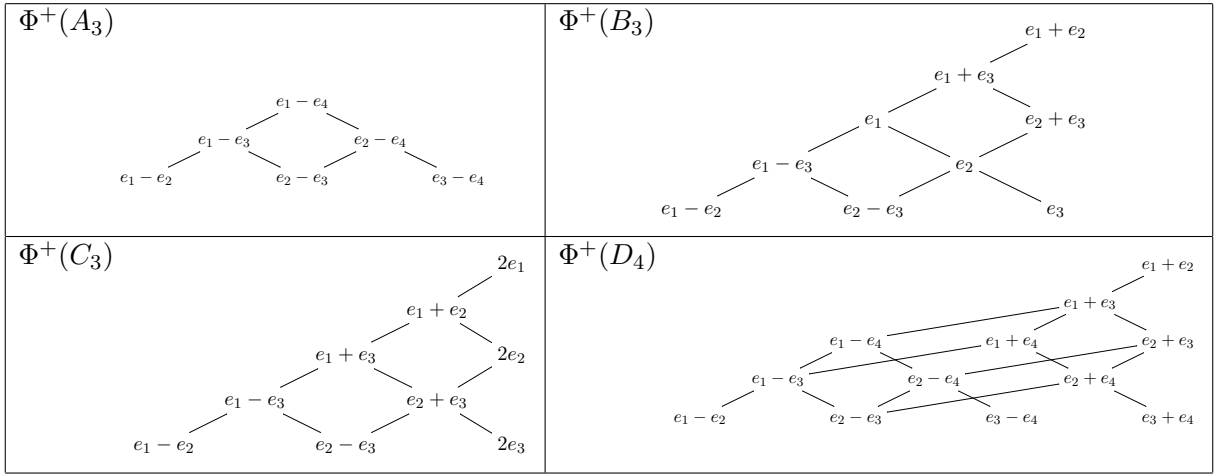

\begin{definition}
The order ideals $J(\Phi^+(W))$ are called the nonnesting partitions of type $W$.
\end{definition}

Let $h$ be the Coxeter number for $W$, let $d_1, d_2, \ldots, d_n$ be the degrees of $W$, and define $Cat(W,q) = \prod_{i=1}^n \frac{[h+d_i]_q}{[d_i]_q}$.  Through work of P.~Cellini, P.~Papi, C.~Athanasiadis, and M.~Haiman~\cite{cellini2002ad,athanasiadis2004generalized,haiman1994conjectures}, it was proven that $Cat(W,1) = |J(\Phi^+ (W))|$.

In 2007, D.\ Panyushev considered applying rowmotion to the nonnesting partitions of type $W$~\cite{panyushev2009orbits}.  He conjectured several properties of rowmotion, the most relevant to our story being the following.

\begin{conjecture}[D.\ Panyushev]
	The order of $\row$ on $J(\Phi^+ (A_n))$, $J(\Phi^+ (D_n))$ for $n$ odd, and $J(\Phi^+ (E_6))$ is $2h$, and the order is $h$ for all other types.
\end{conjecture}

D.\ Bessis and V.\ Reiner then made the stronger conjecture that there was a CSP~\cite{bessis701792cyclic}.

\begin{conjecture}[D. Bessis, V. Reiner]
	Let $C_{2h}$ act on $J(\Phi^+ (W))$ by $\row$.  Then \\ $(J(\Phi^+ (W)), Cat(W,q), C_{2h})$ exhibits the CSP.
\end{conjecture}

These two conjectures were recently proved by D.\ Armstrong, C.\ Stump, and H.\ Thomas in~\cite{Armstrong2011}, in which they inductively defined an equivariant bijection between nonnesting and noncrossing partitions.  By then proving that noncrossing partitions under Kreweras complementation exhibit the CSP (conjectured in~\cite{bessis701792cyclic}), they solved the conjecture.  They further showed that this bijection was uniformly characterized by certain initial conditions, equivariance of rowmotion and Kreweras complementation, and parabolic recursion---thus defining the first uniform bijection between nonnesting and noncrossing partitions.

\begin{theorem}[D.\ Armstrong, C.\ Stump, and H.\ Thomas]
\label{thm:armstrongetal}
	There is a uniformly-characterized equivariant bijection between nonnesting partitions under rowmotion and noncrossing partitions under Kreweras complementation.
\end{theorem}

In the classical types, the three used a known equivariant bijection between noncrossing partitions under Kreweras complementation and noncrossing matchings under rotation in order to have a combinatorial model. 
These noncrossing matchings under rotation are known to have the order conjectured by D.\ Panyushev and to exhibit the CSP, from which the result follows.

It is the noncrossing matchings that we can now associate with linear extensions, using an unpublished result of D.\ White~\cite{rhoades2010cyclic, petersen2009promotion, sagan2010cyclic}.

\begin{theorem}[D.\ White]
\label{thm:WhiteA}
	An equivariant bijection between type $A_n$ noncrossing matchings under rotation and SYT of shape $(n+1,n+1)$ under promotion is given by placing $i$ in the first row when it is the smaller of the two numbers in its matching.
\end{theorem}

In analogy with Theorem \ref{thm:prorowbinomial}, we can restate the type $A_n$ result of Theorem~\ref{thm:armstrongetal} in the language of promotion.  Note that we are unable to rephrase their general result in this way, and we therefore lose the uniformity of their theorem.

\begin{theorem}
\label{thm:prorowcatalan}
	There is an equivariant bijection between $\mathcal{L}([2]\times[n+1])$ under $\pro$ and $J(\Phi^+ (A_n))$ under $\row$.
\end{theorem}

Note that $\mathcal{L}([2]\times[n+1])$ are SYT of shape $(n+1,n+1)$.  Figure~\ref{ex:syt33} illustrates this theorem for $n=2$.

\begin{figure}[ht]
\begin{center}
$\begin{array}{cc}

\scalebox{0.7}{
$\left \{ \quad \begin{gathered}
\setlength{\unitlength}{2mm}
\begin{picture}(12,12)
\bigcircle{5}{5}{5}

\put(5,10){\circle*{0.3}}
\put(9.33,7.5){\circle*{0.3}}
\put(9.33,2.5){\circle*{0.3}}
\put(5,0){\circle*{0.3}}
\put(0.67,2.5){\circle*{0.3}}
\put(0.67,7.5){\circle*{0.3}}

\qbezier[500](5,10)(5,5)(0.67,7.5)
\qbezier[500](9.33,7.5)(5,5)(0.67,2.5)
\qbezier[500](9.33,2.5)(5,5)(5,0)

\put(4.7,10.3){\small{1}}
\put(9.7,7.3){\small{2}}
\put(9.7,2){\small{3}}
\put(4.7,-1.5){\small{4}}
\put(-0.7,2){\small{5}}
\put(-0.7,7.3){\small{6}}

\end{picture} \quad 
\setlength{\unitlength}{2mm}
\begin{picture}(12,12)
\bigcircle{5}{5}{5}

\put(5,10){\circle*{0.3}}
\put(9.33,7.5){\circle*{0.3}}
\put(9.33,2.5){\circle*{0.3}}
\put(5,0){\circle*{0.3}}
\put(0.67,2.5){\circle*{0.3}}
\put(0.67,7.5){\circle*{0.3}}

\qbezier[500](5,10)(5,5)(5,0)
\qbezier[500](9.33,7.5)(5,5)(9.33,2.5)
\qbezier[500](0.67,2.5)(5,5)(0.67,7.5)

\put(4.7,10.3){\small{1}}
\put(9.7,7.3){\small{2}}
\put(9.7,2){\small{3}}
\put(4.7,-1.5){\small{4}}
\put(-0.7,2){\small{5}}
\put(-0.7,7.3){\small{6}}

\end{picture} \quad
\setlength{\unitlength}{2mm}
\begin{picture}(12,12)
\bigcircle{5}{5}{5}

\put(5,10){\circle*{0.3}}
\put(9.33,7.5){\circle*{0.3}}
\put(9.33,2.5){\circle*{0.3}}
\put(5,0){\circle*{0.3}}
\put(0.67,2.5){\circle*{0.3}}
\put(0.67,7.5){\circle*{0.3}}

\qbezier[500](5,10)(5,5)(0.67,7.5)
\qbezier[500](9.33,7.5)(5,5)(0.67,2.5)
\qbezier[500](9.33,2.5)(5,5)(5,0)

\put(4.7,10.3){\small{1}}
\put(9.7,7.3){\small{2}}
\put(9.7,2){\small{3}}
\put(4.7,-1.5){\small{4}}
\put(-0.7,2){\small{5}}
\put(-0.7,7.3){\small{6}}

\end{picture}
\end{gathered} \right \}$} & \scalebox{0.7}{
$\left \{ \quad \begin{gathered}
\setlength{\unitlength}{2mm}
\begin{picture}(12,12)
\bigcircle{5}{5}{5}

\put(5,10){\circle*{0.3}}
\put(9.33,7.5){\circle*{0.3}}
\put(9.33,2.5){\circle*{0.3}}
\put(5,0){\circle*{0.3}}
\put(0.67,2.5){\circle*{0.3}}
\put(0.67,7.5){\circle*{0.3}}

\qbezier[500](5,10)(5,5)(9.33,7.5)
\qbezier[500](9.33,2.5)(5,5)(5,0)
\qbezier[500](0.67,2.5)(5,5)(0.67,7.5)

\put(4.7,10.3){\small{1}}
\put(9.7,7.3){\small{2}}
\put(9.7,2){\small{3}}
\put(4.7,-1.5){\small{4}}
\put(-0.7,2){\small{5}}
\put(-0.7,7.3){\small{6}}

\end{picture} \quad 
\setlength{\unitlength}{2mm}
\begin{picture}(12,12)
\bigcircle{5}{5}{5}

\put(5,10){\circle*{0.3}}
\put(9.33,7.5){\circle*{0.3}}
\put(9.33,2.5){\circle*{0.3}}
\put(5,0){\circle*{0.3}}
\put(0.67,2.5){\circle*{0.3}}
\put(0.67,7.5){\circle*{0.3}}

\qbezier[500](5,10)(5,5)(0.67,7.5)
\qbezier[500](9.33,7.5)(5,5)(9.33,2.5)
\qbezier[500](5,0)(5,5)(0.67,2.5)

\put(4.7,10.3){\small{1}}
\put(9.7,7.3){\small{2}}
\put(9.7,2){\small{3}}
\put(4.7,-1.5){\small{4}}
\put(-0.7,2){\small{5}}
\put(-0.7,7.3){\small{6}}

\end{picture}
\end{gathered} \right \}$} \vspace{10pt} \\

\left \{ \begin{gathered} $$\young(123,456) \hspace{3pt}, \hspace{5pt} \young(125,346)\hspace{3pt},\hspace{5pt} \young(134,256)$$ \end{gathered} \right \} &
\left \{ \begin{gathered} $$\young(135,246)\hspace{3pt},\hspace{5pt} \young(124,356)$$ \end{gathered} \right \} \vspace{10pt} \\

\left \{ \begin{gathered} \scalebox{0.7}{
$$ \xymatrix @-1.2pc {
& \circ & \\
\ar@{-}[ur] \circ & & \ar@{-}[ul] \circ} \text{ } \xymatrix @-1.2pc {
& \circ & \\
\ar@{-}[ur] \bullet & & \ar@{-}[ul] \bullet} \text{ } \xymatrix @-1.2pc {
& \bullet & \\
\ar@{-}[ur] \bullet & & \ar@{-}[ul] \bullet}$$} \end{gathered} \right \} &
\left \{ \begin{gathered} \scalebox{0.7}{$$ \xymatrix @-1.2pc {
& \circ & \\
\ar@{-}[ur] \bullet & & \ar@{-}[ul] \circ} \text{ } \xymatrix @-1.2pc {
& \circ & \\
\ar@{-}[ur] \circ & & \ar@{-}[ul] \bullet}$$} \end{gathered} \right \} \\
\end{array}$
\end{center}
\caption{The two orbits of noncrossing matchings on six points under rotation, the two orbits of SYT of shape $(3,3)$ under $\pro$, and the two orbits of $J(\Phi^+(A_2))$ under $\row$.}
\label{ex:syt33}
\end{figure}
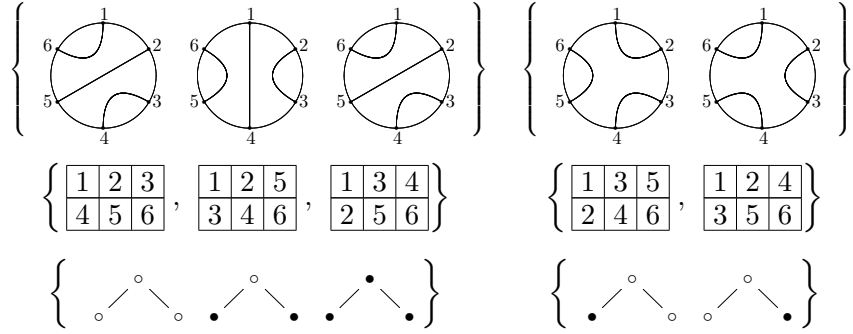

Our main theorem, Theorem~\ref{thm:maintheorem}, gives an equivariant bijection between promotion and rowmotion on the order ideals of any poset with rows and columns (in a sense we will make precise in Section~\ref{sec:rcposets}).  In particular, the result holds for \emph{all} skew SYT with at most two rows, and so we obtain Theorems~\ref{thm:prorowbinomial} and~\ref{thm:prorowcatalan} as special cases.

\section{Machinery}
\label{sec:machinery}

In this section, we develop the machinery needed to prove our main theorem.  We first recall P.\ Cameron and D.~Fon-der-Flaass's permutation group on the order ideals of a poset, which we call the toggle group.  We then define \rcposet{}s, interpret promotion and rowmotion as elements in the toggle group of an \rcposet{}, and show that promotion and rowmotion are conjugate elements in these toggle groups.  The following lemma then specifies an equivariant bijection between the order ideals of \rcposet{}s under promotion and the order ideals of \rcposet{}s under rowmotion.

\begin{lemma}
\label{lem:eq}
	Let $G$ be a group acting on a set $X$, and let $g_1$ and $g_2 = g g_1 g^{-1}$ be conjugate elements in $G$.  Then $x \to gx$ gives an equivariant bijection between $X$ under $\langle g_1 \rangle$ and $X$ under $\langle g_2 \rangle$.
\end{lemma}

This lemma is described by the following commutative diagram.

$$\xymatrix{
X \ar [r]^{g} \ar [d]_{g_1} & X \ar [d]^{g_2=g g_1 g^{-1}} \\
X \ar [r]_{g} & X
}$$

\subsection{The Toggle Group}
\label{sec:togglegroup}

Let $\mathcal{P}$ be a poset and let $J(\mathcal{P})$ be its set of order ideals.  In~\cite{cameron1995orbits}, P.\ Cameron and D.\ Fon-der-Flaass defined a group acting on $J(\mathcal{P})$.

\begin{definition}[P.\ Cameron and D.\ Fon-der-Flaass] \rm
For each $q \in \mathcal{P}$, define $t_p: J(\mathcal{P}) \to J(\mathcal{P})$ to act by toggling $p$ if possible.  That is, if $I \in J(\mathcal{P})$,

$$t_p (I) = \left\{
	\begin{array}{ll}
		I \cup \{p\} & \text{ if } p \notin I \text{ and if } p' < p \text{ then } p' \in I,\\
		I-p & \text{ if } p \in I \text{ and if } p' > p \text{ then } p' \notin I,\\
		I & \text{ otherwise}.\\
	\end{array} \right.
$$
\end{definition}

\begin{definition}[P.\ Cameron and D.\ Fon-der-Flaass] \rm
	The toggle group $T(\mathcal{P})$ of a poset $\mathcal{P}$ is the subgroup of the permutation group $\mathfrak{S}_{J(\mathcal{P})}$ generated by $\{ t_p\}_{p\in \mathcal{P}}$.
\end{definition}

Note that $T(\mathcal{P})$ has the following obvious relations (which do not constitute a full presentation).
\begin{enumerate}
	\item $t_p^2 = 1,$ and
	\item $(t_p t_{p'})^2 = 1$ if $p$ and $p'$ do not have a covering relation.
\end{enumerate}

P.\ Cameron and D.\ Fon-der-Flaass characterized rowmotion as an element of $T(\mathcal{P})$.

\begin{theorem}[P.\ Cameron and D.\ Fon-der-Flaass]
\label{thm:linextrow}
Fix a linear extension $\mathcal{L}$ of $\mathcal{P}$.  Then
$$t_{\mathcal{L}^{-1} (1)} t_{\mathcal{L}^{-1} (2)} \cdots t_{\mathcal{L}^{-1}
(n)}$$ acts as $\row$.
\end{theorem}

\subsection{Rowed-and-Columned Posets}
\label{sec:rcposets}

We now define \rcposet{}s---certain posets with elements that neatly fit into rows and columns and with covering relations allowed only between diagonally adjacent elements.  We will interpret promotion as an action that toggles the columns of order ideals of \rcposet{}s, and rowmotion as an action that toggles the rows.

\begin{definition} \rm
Let $\Pi \subset \mathbb{R}^2$ be the set of points in the integer span of $(2,0)$ and $(1,1)$.  A \emph{rowed-and-columned (rc) poset} $\mathcal{R}$ is a finite poset together with a map $\pi: \mathcal{R} \to \Pi$, where if $p_1,p_2 \in \mathcal{R}$, $p_1$ covers $p_2$, and $\pi(p_1)=(i,j)$, then $\pi(p_2)=(i+1,j-1)$ or $\pi(p_2)=(i-1,j-1)$.  For $p \in \mathcal{R}$, we call $\pi(p) \in \Pi$ the \emph{position} of $p$. 
\end{definition}

Let the \emph{height} $h$ of an \rcposet{} be the maximum number of elements in a single position $(i,j)$.  The $j$th \emph{row} of an \rcposet{} $\mathcal{R}$ is the set of elements of $\mathcal{R}$ in positions $\{(i,j)\}_{i}$.  The $i$th \emph{column} of an \rcposet{} is the set of elements of $\mathcal{R}$ in positions $\{(i,j)\}_j$.  Let $n$ denote the maximal non-empty row and $k$ the maximal non-empty column.  Without loss of generality---purely for ease of notation---we may assume that the \rcposet{} is translated into the first quadrant so that the rows are labeled from $1$ to $n$ and the columns from $1$ to $k$.  For an example, see Figure~\ref{ex:rc}.

\begin{figure}[ht]
\begin{center}

$\begin{gathered}\includegraphics[scale=0.35]{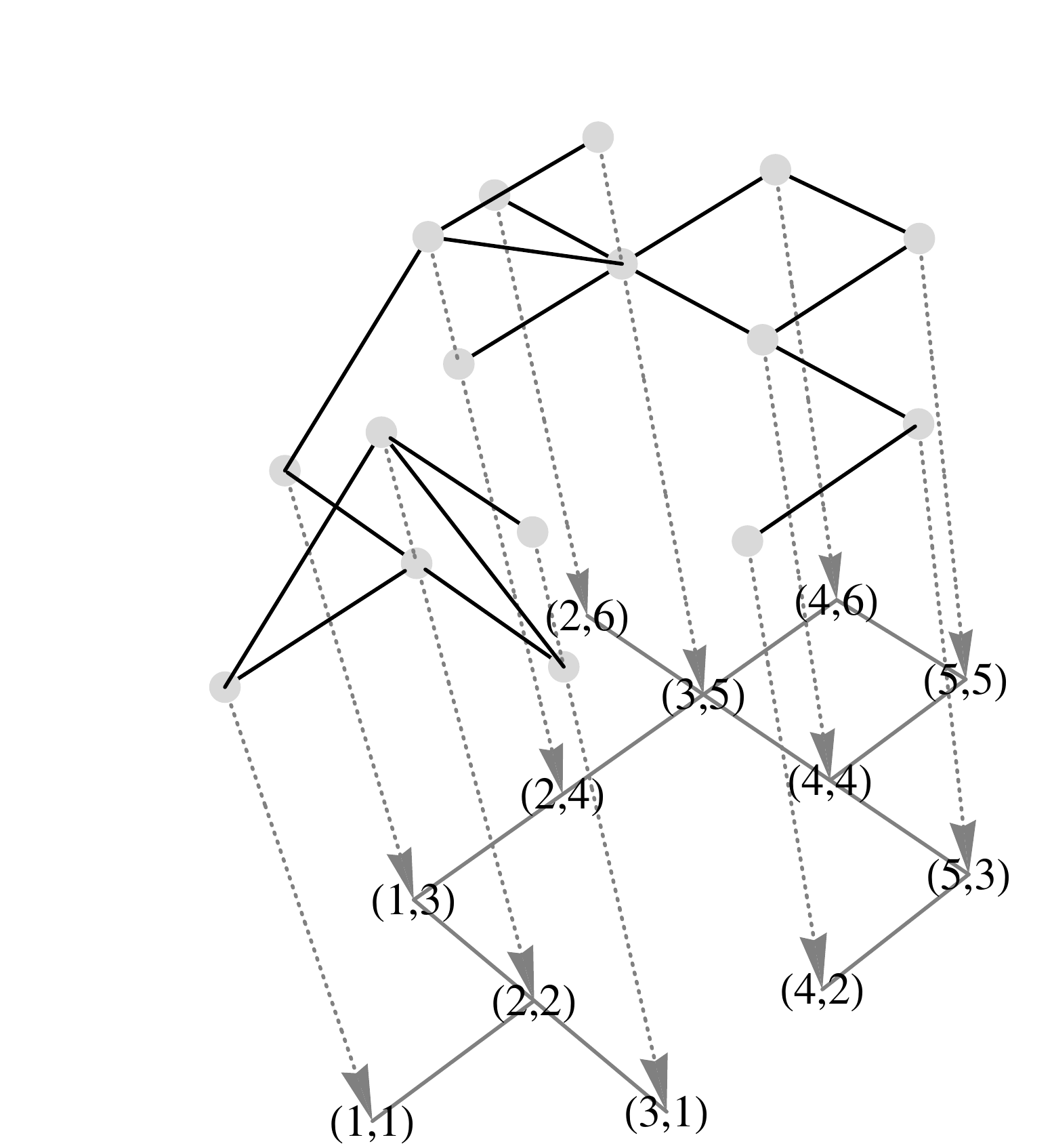}\end{gathered}$
\end{center}
\caption{\small
This figure represents an \rcposet{} with height $h=2$, $k=5$ columns, and $n=6$ rows.  When there are two elements with the same position, the second element is raised.  The position map $\pi$ is indicated by a dotted arrow down.  Covering relations are drawn with solid black lines and are projected down as solid gray lines.}
\label{ex:rc}
\end{figure}

By definition, an element is in a row and column of the same parity.  This is the key to our proof of Theorem~\ref{thm:maintheorem}.

\begin{example} \rm
Other examples of \rcposet{}s of height one are:
\begin{enumerate}
	\item $[n] \times [k]$ (take $\pi((i,j))=(i-j,i+j)$),
	\item $\Phi^+ (A_n)$ (take $\pi(e_i-e_j)=(i+j,j-i)$), and
	\item $\Phi^+ (B_n) \cong \Phi^+ (C_n)$ (for $\Phi^+(B_n)$, take $\pi(e_i-e_j)=(i+j,j-i)$, $\pi(e_i)=(n+1+i,n+1-i)$, and for $i<j$ let $\pi(e_i+e_j)=(2n+2-(j-i),2n+2-(i+j))$).
\end{enumerate}
\end{example}

We will consider certain posets of height one in Section~\ref{sec:othertypes}.  We remark here that $\Phi^+(D_n)$ can be drawn as an \rcposet{} of height two (see Section~\ref{sec:threedim}).

On an \rcposet{} $\mathcal{R}$, promotion can be defined as an action that scans across the \emph{columns} of $\mathcal{R}$, and rowmotion as an action that scans down the \emph{rows}.  By the commutation relations of the toggle group, the order that we take the elements within a row or column does not matter.

\begin{definition} \rm
\label{def:rici}
	If $\mathcal{R}$ is an \rcposet{}, let $r_i = \prod t_p$, where the product is over all elements in row $i$ and let $c_i = \prod t_p$, where the product is over all elements in column $i$.
\end{definition}

Then, since no elements within a row or column of an \rcposet{} share a covering relation, the following relations hold.
\begin{enumerate}
\item $r_i^2 = c_i^2 = 1,$ and
\item if $|i-j|>1$, $(r_i r_j)^2 = (c_i c_j)^2 = 1$.
\end{enumerate}

\subsection{Promotion and Rowmotion in the Toggle Group}
\label{sec:prorowtoggle}

We interpret promotion and rowmotion as elements of the toggle group of an \rcposet{} with $n$ rows and $k$ columns.

\begin{definition} \rm
\label{def:prorowtogg}
	\begin{enumerate}
		\item Given $\nu \in \mathfrak{S}_k$ let $\pro_{\nu} = \prod_{i=1}^k c_{\nu(i)} =  c_{\nu(1)}\cdot c_{\nu(2)}\cdots c_{\nu(k)}$.
		\item Likewise, given $\omega \in \mathfrak{S}_n$ let $\row_{\omega} = \prod_{i=1}^n r_{\omega(i)}$.
\end{enumerate}
\end{definition}

We now specify the element of the toggle group that we will take to act as rowmotion.

\begin{corollary}
	On an \rcposet{}, $\row_{12\ldots n}$ acts as $\row$.
\end{corollary}

\begin{proof}
	This follows immediately from Theorem \ref{thm:linextrow}.
\end{proof}

Interpreting promotion as an element of the toggle group takes slightly more work.  Let $\mathcal{P}$ be a Ferrers diagram.  Following R.\ Stanley in~\cite{stanley2009promotion}, we define promotion using the order ideals $J(\mathcal{P})$.  Linear extensions $\mathcal{L}$ can be interpreted as maximal chains $\emptyset = I_0 \subset I_1 \cdots \subset I_n = \mathcal{P}$ in $J(\mathcal{P})$ by taking $\mathcal{L}(p)=i$ if $p$ is the element in the singleton set $I_{i+1} - I_{i}$.

For example,
\tiny
$\young(123,456)$
\normalsize
corresponds to the chain
$\begin{gathered}
\scalebox{0.7}{$\xymatrix @-1.2pc {
 & & & \text{\tiny \young(123,456)} \\
 & & \text{\tiny \young(123,45)} \ar@{=}^-{6}[ur] & \\
& \text{\tiny \young(123,4)}  \ar@{=}^-{5}[ur] & & \text{\supertiny \yng(2,2)}
\ar@{-}[ul] \\
\text{\tiny \young(123)}  \ar@{=}^-{4}[ur] & & \text{\supertiny \yng(2,1)}
\ar@{-}[ul]  \ar@{-}[ur] & \\
 & \text{\tiny \young(12)}  \ar@{=}^-{3}[ul]  \ar@{-}[ur] & & \text{\supertiny
\yng(1,1)}  \ar@{-}[ul] \\
 & & \text{\tiny \young(1)} \ar@{=}^-{2}[ul] \ar@{-}[ur] & \\
 & & & \ar@{=}[ul]^-{1} \emptyset}$}\end{gathered}$
in $J \left(\mbox{\tiny $\yng(3,3)$}
\right)$.

The promotion of $\lambda = (\emptyset = I_0 \subset I_1 \cdots \subset I_n = \mathcal{P})$ is $\tau_{n-1} \cdots \tau_1 \lambda$, where $\tau_i$ acts on a chain by switching $I_i$ to the other order ideal in the interval $[I_{i-1}, I_{i+1}]$, if one exists.  Figure~\ref{ex:max33} illustrates promotion on the maximal chains.

When $\lambda / \mu = (n+k,m) / (k)$ is a skew Ferrers diagram with at most two rows, we can draw the Hasse diagram of $J(\lambda / \mu)$ as an \rcposet{}.  The $i$th step of a maximal chain in $J(\lambda / \mu)$ is taken to be northwest if the corresponding linear extension of $\lambda / \mu$ associates $i$ to an element in the first row, and northeast otherwise.  We take advantage of this planarity with the following definition.

\begin{definition} \rm
\label{def:interior}
	If $\lambda / \mu = (n+k,m) / (k)$ is a skew Ferrers diagram with at most two rows, define the interior $\text{Int}(J(\lambda / \mu))$ to be the \rcposet{} with elements the boxes of $J(\lambda / \mu)$ and covering relations between two elements when their corresponding boxes are adjacent.
\end{definition}

When $\lambda / \mu$ has at most two rows, a maximal chain in $J(\lambda / \mu)$ traces out an order ideal---defined by the boxes to the right of the maximal chain---in $\text{Int}(J(\lambda / \mu))$.

\begin{example} \rm
The Hasse diagram of $J \left(\mbox{\tiny $\yng(3,3)$} \right)$ (with the boxes of the Hasse diagram labeled by $a, b,$ and $c$) is
$\begin{gathered}
\scalebox{0.5}{$
\xymatrix @-1.0pc {
 & & & \circ \\
 & & \circ \ar@{-}[ur] & \\
& \circ  \ar@{-}[ur] & \bf{a} & \circ \ar@{-}[ul] \\
\circ \ar@{-}[ur] & \bf{b} & \circ \ar@{-}[ul]  \ar@{-}[ur] & \\
 & \circ  \ar@{-}[ul]  \ar@{-}[ur] & \bf{c} & \circ  \ar@{-}[ul] \\
 & & \circ \ar@{-}[ul] \ar@{-}[ur] & \\
 & & & \ar@{-}[ul] \circ}$}
\end{gathered}$.
Therefore, $\text{Int}\left(J \left(\mbox{\tiny $\yng(3,3)$} \right) \right)$ is
$\begin{gathered}
\scalebox{0.7}{
$\xymatrix @-1.2pc {
& \bf{b} & \\
\ar@{-}[ur] \bf{c} & & \ar@{-}[ul] \bf{a}}$}
\end{gathered}$.
\end{example}

Figure~\ref{ex:max33} illustrates the bijection from SYT of shape $(3,3)$ to order ideals of $\Phi^+(A_2)$.

\begin{figure}
\begin{center}
$\begin{array}{cc}
\left \{ \begin{gathered} $$\young(123,456) \hspace{3pt}, \hspace{5pt} \young(125,346)\hspace{3pt},\hspace{5pt} \young(134,256)$$ \end{gathered} \right \} &
\left \{ \begin{gathered} $$\young(135,246)\hspace{3pt},\hspace{5pt} \young(124,356)$$ \end{gathered} \right \} \vspace{10pt} \\

\left \{ \begin{gathered}\scalebox{0.6}{$
\xymatrix @-1.0pc {
 & & & \circ \\
 & & \circ \ar@{=}^-{6}[ur] & \\
& \circ  \ar@{=}^-{5}[ur] & & \circ \ar@{-}[ul] \\
\circ \ar@{=}^-{4}[ur] & & \circ \ar@{-}[ul]  \ar@{-}[ur] & \\
 & \circ  \ar@{=}^-{3}[ul]  \ar@{-}[ur] & & \circ  \ar@{-}[ul] \\
 & & \circ \ar@{=}^-{2}[ul] \ar@{-}[ur] & \\
 & & & \ar@{=}^-{1}[ul] \circ}$}
\scalebox{0.6}{$
\xymatrix @-1.0pc {
 & & & \circ \\
 & & \circ \ar@{=}^-{6}[ur] & \\
& \circ  \ar@{-}[ur] & & \circ \ar@{=}^-{5}[ul] \\
\circ \ar@{-}[ur] & & \circ \ar@{-}[ul]  \ar@{=}^-{4}[ur] & \\
 & \circ  \ar@{-}[ul]  \ar@{=}^-{3}[ur] & & \circ  \ar@{-}[ul] \\
 & & \circ \ar@{=}^-{2}[ul] \ar@{-}[ur] & \\
 & & & \ar@{=}^-{1}[ul] \circ}$}
 \scalebox{0.6}{$
\xymatrix @-1.0pc {
 & & & \circ \\
 & & \circ \ar@{=}^-{6}[ur] & \\
& \circ  \ar@{=}^-{5}[ur] & & \circ \ar@{-}[ul] \\
\circ \ar@{-}[ur] & & \circ \ar@{=}^-{4}[ul]  \ar@{-}[ur] & \\
 & \circ  \ar@{-}[ul]  \ar@{-}[ur] & & \circ  \ar@{=}^-{3}[ul] \\
 & & \circ \ar@{-}[ul] \ar@{=}^-{2}[ur] & \\
 & & & \ar@{=}^-{1}[ul] \circ}$}\end{gathered} \right \} &
 
\left \{ \begin{gathered}\scalebox{0.6}{$
\xymatrix @-1.0pc {
 & & & \circ \\
 & & \circ \ar@{=}^-{6}[ur] & \\
& \circ  \ar@{-}[ur] & & \circ \ar@{=}^-{5}[ul] \\
\circ \ar@{-}[ur] & & \circ \ar@{-}[ul]  \ar@{=}^-{4}[ur] & \\
 & \circ  \ar@{-}[ul]  \ar@{-}[ur] & & \circ  \ar@{=}^-{3}[ul] \\
 & & \circ \ar@{-}[ul] \ar@{=}^-{2}[ur] & \\
 & & & \ar@{=}^-{1}[ul] \circ}$}
 \scalebox{0.6}{$
\xymatrix @-1.0pc {
 & & & \circ \\
 & & \circ \ar@{=}^-{6}[ur] & \\
& \circ  \ar@{=}^-{5}[ur] & & \circ \ar@{-}[ul] \\
\circ \ar@{-}[ur] & & \circ \ar@{=}^-{4}[ul]  \ar@{-}[ur] & \\
 & \circ  \ar@{-}[ul]  \ar@{=}^-{3}[ur] & & \circ  \ar@{-}[ul] \\
 & & \circ \ar@{=}^-{2}[ul] \ar@{-}[ur] & \\
 & & & \ar@{=}^-{1}[ul] \circ}$}\end{gathered} \right \} \vspace{10pt} \\

 \left \{ \begin{gathered}\scalebox{0.7}{
$\xymatrix @-1.2pc {
& \bullet & \\
\ar@{-}[ur] \bullet & & \ar@{-}[ul] \bullet}$
$\xymatrix @-1.2pc {
& \circ & \\
\ar@{-}[ur] \bullet & & \ar@{-}[ul] \circ}$
$\xymatrix @-1.2pc {
& \circ & \\
\ar@{-}[ur] \circ & & \ar@{-}[ul] \bullet}$}\end{gathered} \right \} &
\left \{ \begin{gathered}\scalebox{0.7}{$\xymatrix @-1.2pc {
& \circ & \\
\ar@{-}[ur] \circ & & \ar@{-}[ul] \circ}$
$\xymatrix @-1.2pc {
& \circ & \\
\ar@{-}[ur] \bullet & & \ar@{-}[ul] \bullet}$}\end{gathered} \right \} \\

\end{array}$
\end{center}
\caption{The two orbits of SYT of shape $(3,3)$ under promotion, the same two orbits as maximal chains, and the same two orbits as order ideals under $\pro$.}
\label{ex:max33}
\end{figure}
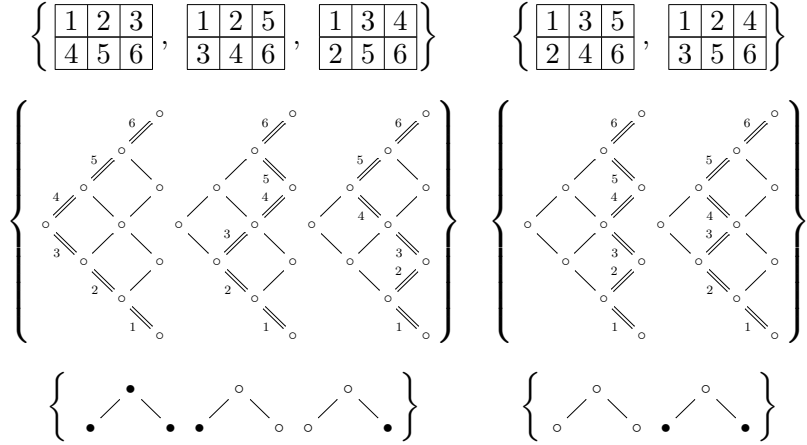

\begin{theorem}
\label{prop:proback}
Let $\mathcal{P}=(n+k,m) / (k)$ be a skew Ferrers diagram with at most two rows.  Then there is an equivariant bijection from $\text{Int}(J(\mathcal{P}))$ under $\pro_{k\ldots21}$ to SYT of shape $\mathcal{P}$ under ordinary promotion.
\end{theorem}

\begin{proof}
This follows from the characterization of promotion as an action on maximal chains in $J(\mathcal{P})$.
\end{proof}


We now extend the definition of promotion from order ideals of \rcposet{}s that correspond to skew SYT to order ideals of arbitrary \rcposet{}s.

\begin{definition}
Given an rc-poset $\mathcal{R}$ and an order ideal $I\in J(R)$, define the promotion of $I$ to be $\pro(I)=\pro_{k\ldots21}(I)$.
\end{definition}

We also generalize the maximal chains of the above discussion to height one rc-posets by defining \emph{boundary paths}.

\begin{definition} \rm
\label{def:boundarypath}
We define the \emph{boundary path} of an order ideal of a connected \rcposet{} of height one to be the path that separates the order ideal from the rest of the poset.  We encode boundary paths as binary words by writing a $1$ for a northeast step and a $0$ for a southeast step.
\end{definition}

When we start with the poset $\Phi^+(A_n)$ under $\pro$ and map to SYT, we can apply Theorem~\ref{thm:WhiteA} to obtain noncrossing matchings under rotation.  In this language, $i$ is the smaller number in its partition if the $i$th step of the boundary path is northeast.

We apply this idea of boundary paths under $\pro$ to noncrossing objects under rotation in Section~\ref{sec:othertypes}, and generalize it in Section~\ref{sec:threedim}.  In Sections~\ref{sec:threedim} and~\ref{sec:asmtsscpp}, we consider generalizations to the type $D_n$ positive root poset, plane partitions, the ASM poset, and the TSSCPP poset.

\section{The Conjugacy of Promotion and Rowmotion}
\label{sec:conjugate}

We now prove that promotion and rowmotion are conjugate elements in the toggle group of an \rcposet{}.  We spend the rest of the paper applying this theorem.

\begin{lemma}[\cite{humphreys1992reflection}]
\label{lem:symconj}
	Let $G$ be a group 
whose generators
$g_1,\ldots,g_n$ 
satisfy
$g_i^2=1$ and $(g_i g_j)^2 = 1$ if $|i-j|>1$.  Then for any $\omega, \nu \in \mathfrak{S}_n$, $\prod_i g_{\omega(i)}$ and $\prod_i g_{\nu(i)}$ are conjugate.
\end{lemma}

\begin{proof}
	The proof is constructive.  It suffices to show that $\prod_{j=1}^n g_i$ is conjugate to any $\prod_i g_{\nu(i)}$.  Find the largest number $k$ that can be pushed to the left using the commutation relations, and then conjugate by $k$ and again use the commutation relations to push $k$ as far to the left as possible.  Then either $k$ has been stopped by $k+1$, or $k$ has been stopped by $k-1$.  If it has been stopped by $k-1$, the number of inversions has been decreased by at least 1.  If it has been stopped by $k+1$, then the inversion $(k,k-1)$ has been replaced by $(k+1,k)$.  Thus, with each step either the number of inversions decreases or the numbers in an inversion increase and so this procedure terminates with $\prod_{j=1}^n g_i$.
\end{proof}

\begin{theorem}
\label{thm:maintheorem}
For any \rcposet{} $\mathcal{R}$ and any $\omega \in \mathfrak{S}_n$ and $\nu \in \mathfrak{S}_k$, there is an equivariant bijection between $J(\mathcal{R})$ under $\pro_{\nu}$ and $J(\mathcal{R})$ under $\row_{\omega}$.
\end{theorem}

\begin{proof}
Since the row (resp. column) toggles $r_i$ (resp. $c_i$) satisfy the conditions of Lemmas~\ref{lem:symconj} and~\ref{lem:eq}, for any \rcposet{} $\mathcal{R}$ and any $\omega, \nu \in \mathfrak{S}_n$ (resp. $\mathfrak{S}_k$), there is an equivariant bijection between $J(\mathcal{R})$ under $\row_{\omega}$ (resp. $\pro_{\omega}$) and $J(\mathcal{R})$ under $\row_{\nu}$ (resp. $\pro_{\nu}$).

Therefore, we may restrict to considering only $\row_{135\ldots246\ldots}$ and $\pro_{135\ldots246\ldots}$.  But since all $t_p$ with $p$ in an odd (resp. even) column or row commute with one another, and since elements in an odd (resp. even) row are also necessarily in an odd (resp. even) column, we conclude that $\row_{135\ldots246\ldots}$ is equal to $\pro_{135\ldots246\ldots}.$
\end{proof}

We may further ask for an explicit equivariant bijection from rowmotion $\row=\row_{12\ldots n}$ to promotion $\pro=\pro_{k \ldots 21}$.  It turns out to be easier to conjugate $\row^{-1}=\row_{n\ldots21}$ to $\pro_{k \ldots 21}$.

Define the $j$th \emph{diagonal} of an \rcposet{} to be the set of elements in positions $\{(2(j-1)+i,i)\}_i$ that lie in $\mathcal{R}$.  Let $m$ be the maximal non-empty diagonal.

\begin{definition} \rm
	If $\mathcal{R}$ is an \rcposet{}, define $d_j = \prod t_p$, where the product is over all elements in diagonal $j$.  The order within a diagonal does matter, and we specify the order of the elements to be (from left to right) from row with smallest index to row with largest index.
\end{definition}

\begin{theorem}
\label{thm:explicit}
	An equivariant bijection from $J(\mathcal{R})$ under $\row^{-1}=\row_{n \ldots 21}$ to $J(\mathcal{R})$ under $\pro_{k \ldots 21}$ is given by acting on an order ideal by $D = \prod_{i=1}^{m-1} \prod_{j=i}^1 d_j^{-1}.$
\end{theorem}

\begin{proof}

	It is immediate from the commutation relations of toggles in the toggle group that $\row_{n \ldots 21} = \prod_{i=1}^m d_i = d_1 d_2 \cdots d_m$ and $\pro_{k \ldots 21} = \prod_{i=m}^{1} d_i = d_m d_{m-1} \cdots d_1$.  Then the construction in the proof for Lemma~\ref{lem:symconj} gives us the element $D$.

\end{proof}

Note that this theorem implies Theorem~\ref{thm:maintheorem}.  We thank an anonymous referee for a simplification of this proof.  An example of this construction is given in the following commutative diagram.

$$\xymatrix{
 \begin{gathered}\includegraphics[scale=1]{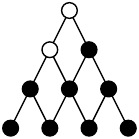}\end{gathered} \ar [r]^{D} \ar [d]_{\row^{-1}} & \begin{gathered}\includegraphics[scale=1]{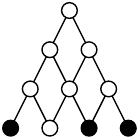}\end{gathered} \ar [d]^{\pro=D \row^{-1} D^{-1}} \\
\begin{gathered}\includegraphics[scale=1]{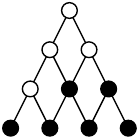}\end{gathered} \ar [r]_{D} & \begin{gathered}\includegraphics[scale=1]{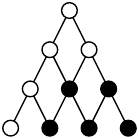}\end{gathered}
}$$

\section{RC-Posets of Height $1$}
\label{sec:othertypes}

While rowmotion on an \rcposet{} can be difficult to analyze, promotion often has a simple equivariant bijection with a known combinatorial object under rotation. We apply Theorem~\ref{thm:maintheorem} over the following two sections to obtain such bijections, from which we obtain cyclic sieving phenomena as corollaries.  

In this section, we investigate the following \rcposet{}s of height one: $[n] \times [k]$, $J([2] \times [n-1])$, $\Phi^+(A_n)$, and $\Phi^+(B_n) \cong \Phi^+(C_n)$.

\subsection{\boldmath $[n]\times[k]$}
\label{sec:nxk}

As a corollary of Theorem~\ref{thm:maintheorem} we obtain a new proof of Theorem~\ref{thm:prorowbinomial}.

\begin{proof}[Proof of Theorem~\ref{thm:prorowbinomial}]
By the reasoning in Section \ref{sec:prorowtoggle}, $\mathcal{L}([n] \oplus [k])$ under $\pro$ is in equivariant bijection with $J([n]\times[k])$ under $\pro$.  The result then follows from Theorem \ref{thm:maintheorem}.
\end{proof}

Since $\mathcal{L}([n] \oplus [k])$ under $\pro$ is in bijection with $\binom{[n+k]}{k}$ under the cycle $(1,2,\ldots,n+k)$, we can restate the theorem using the map from Theorem~\ref{prop:proback}.

\begin{theorem}
\label{thm:square}
There is an equivariant bijection between $I \in J([n] \times [k])$ under $\row$ and binary words of the form $w(I)$ under rotation, where $w(I)=w_1 w_2 \ldots w_{n+k}$ is a binary word of length $n+k$ with $n$ 1's.
\end{theorem}

The bijection is given by using our bijection from $J([n] \times [k])$ under $\row$ to $J([n] \times [k])$ under $\pro$, and then setting $w_i$ to 1 if the $i$th step of the boundary path is northeast, and to 0 otherwise.

A CSP follows immediately from Theorem~\ref{thm:denni}.  An example is given in Figures~\ref{ex:example34} and~\ref{ex:example342}.

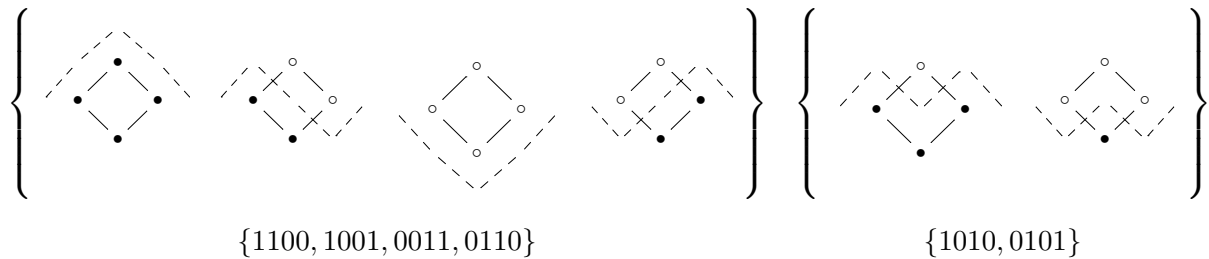
\begin{figure}[ht]
\begin{center}
$\begin{array}{cc}
 \left \{ \begin{gathered}\scalebox{0.7}{$$\xymatrix @-1.2pc {
& & \ar@{--}[dr] & \\
&\ar@{--}[ur] & \bullet &\ar@{--}[dr]  &  \\
\ar@{--}[ur] & \ar@{-}[ur] \bullet & & \ar@{-}[ul] \bullet  &  \\
& & \bullet \ar@{-}[ur] \ar@{-}[ul] & & \\
& & &  & } \text{ } \xymatrix @-1.2pc {
& & & &  \\
& \ar@{--}[dr] & \circ & &  \\
\ar@{--}[ur] & \ar@{-}[ur] \bullet & \ar@{--}[dr] & \ar@{-}[ul] \circ &  \\
& & \bullet \ar@{-}[ur] \ar@{-}[ul] & \ar@{--}[ur] & \\
& & &  & } \text{ } \xymatrix @-1.0pc {
& & &  & \\
& & \circ &  & \\
\ar@{--}[dr] & \ar@{-}[ur] \circ & & \ar@{-}[ul] \circ &  \\
& \ar@{--}[dr] & \circ \ar@{-}[ur] \ar@{-}[ul] & \ar@{--}[ur] & \\
& & \ar@{--}[ur] &  & } \text{ } \xymatrix @-1.2pc {
& & &  & \\
& & \circ & \ar@{--}[dr] & \\
\ar@{--}[dr] & \ar@{-}[ur] \circ & \ar@{--}[ur] & \ar@{-}[ul] \bullet &  \\
& \ar@{--}[ur] & \bullet \ar@{-}[ur] \ar@{-}[ul] &  & \\
& & &  & }$$}\end{gathered} \right \} & \left \{ \begin{gathered}\scalebox{0.7}{$$ \xymatrix @-1.0pc {
& & &  & \\
& \ar@{--}[dr] & \circ & \ar@{--}[dr] &  \\
\ar@{--}[ur] & \ar@{-}[ur] \bullet & \ar@{--}[ur] & \ar@{-}[ul] \bullet &  \\
& & \bullet \ar@{-}[ur] \ar@{-}[ul] & & \\
& & &  & } \text{ } \xymatrix @-1.2pc {
& & & &  \\
& & \circ & &  \\
\ar@{--}[dr] & \ar@{-}[ur] \circ & \ar@{--}[dr] & \ar@{-}[ul] \circ &  \\
& \ar@{--}[ur] & \bullet \ar@{-}[ur] \ar@{-}[ul] & \ar@{--}[ur] & \\
& & &  & }$$}\end{gathered} \right \}  \vspace{10pt}\\
\left \{ 1100, 1001, 0011, 0110 \right \}
 &
 \left \{1010, 0101 \right \}
\end{array}$
\end{center}
\caption{
The two orbits of $J([2] \times [2])$ under $\pro$ (the dashed lines are the boundary paths corresponding to the order ideals) and the two orbits of binary words of length 4 under rotation (obtained from the boundary paths).
}
\label{ex:example342}
\end{figure}

\subsection{\boldmath $J([2] \times [n-1])$}
\label{sec:nxkb}

Observe that $J([2] \times [n-1])$ can be embedded as the left half of $[n] \times [n]$.  It is not hard to see that the map from Theorem~\ref{prop:proback} can be adapted
to these boundary paths.

\begin{theorem}
\label{thm:halfsquare}
There is an equivariant bijection between $I \in J(J([2] \times [n-1]))$ under $\row$ and binary words of the form $w(I) (1-w(I))$ under rotation, where
$w(I)=w_1 w_2 \ldots w_n$ is any binary word of length $n$, and $1-w(I)$ is the word of length $n$ whose $i$th letter equals $1-w_i$.
\end{theorem}

Again, we first use our bijection from $J(J([2] \times [n-1]))$ under $\row$ to $J(J([2] \times [n-1]))$ under $\pro$, and then set $w_i$ equal to 1 if the
$i$th step of the boundary path is northeast, and 0 otherwise. 
Promotion acts on these boundary paths in exactly the same way as it did on the maximal chains in Theorem~\ref{prop:proback}.
 This theorem is illustrated for the case $n=3$ in Figure~\ref{ex:halfsquare}.

\begin{figure}[ht]
\begin{center}
$\begin{array}{cc}
\left \{ \begin{gathered}\scalebox{0.6}{$
\xymatrix @-1.0pc {
 & & & \\
& & & \circ \\
& & \circ  \ar@{-}[ur] & \\
\ar@{--}[dr] & \circ  \ar@{-}[ur] & & \circ  \ar@{-}[ul] \\
 & \ar@{--}[dr] & \circ \ar@{-}[ul] \ar@{-}[ur] & \\
 & & \ar@{--}[dr] & \ar@{-}[ul] \circ \\
 & & &}$
 $\xymatrix @-1.0pc {
  & & & \\
& & & \circ \\
& & \circ  \ar@{-}[ur] & \\
\ar@{--}[dr] & \circ  \ar@{-}[ur] & & \circ  \ar@{-}[ul] \\
 & \ar@{--}[dr] & \circ \ar@{-}[ul] \ar@{-}[ur] & \\
 & & \ar@{--}[ur]& \ar@{-}[ul] \bullet \\
 & & &}$
 $\xymatrix @-1.0pc {
  & & & \\
& & & \circ \\
& & \circ  \ar@{-}[ur] & \\
\ar@{--}[dr] & \circ  \ar@{-}[ur] & \ar@{--}[ur] & \bullet  \ar@{-}[ul] \\
 & \ar@{--}[ur] & \bullet \ar@{-}[ul] \ar@{-}[ur] & \\
 & & & \ar@{-}[ul] \bullet \\
 & & &}$
 $\xymatrix @-1.0pc {
 & & & \\
& & \ar@{--}[ur] & \bullet \\
& \ar@{--}[ur] & \bullet  \ar@{-}[ur] & \\
\ar@{--}[ur] & \bullet  \ar@{-}[ur] & & \bullet  \ar@{-}[ul] \\
 & & \bullet \ar@{-}[ul] \ar@{-}[ur] & \\
 & & & \ar@{-}[ul] \bullet \\
 & & &}$
 $\xymatrix @-1.0pc {
  & & & \\
& & \ar@{--}[dr] & \circ \\
& \ar@{--}[ur] & \bullet  \ar@{-}[ur] & \\
 \ar@{--}[ur] & \bullet  \ar@{-}[ur] & & \bullet  \ar@{-}[ul] \\
 & & \bullet \ar@{-}[ul] \ar@{-}[ur] & \\
 & & & \ar@{-}[ul] \bullet \\
 & & &}$
 $\xymatrix @-1.0pc {
  & & & \\
& & & \circ \\
& \ar@{--}[dr] & \circ  \ar@{-}[ur] & \\
\ar@{--}[ur] & \bullet  \ar@{-}[ur] & \ar@{--}[dr] & \circ  \ar@{-}[ul] \\
 & & \bullet \ar@{-}[ul] \ar@{-}[ur] & \\
 & & & \ar@{-}[ul] \bullet \\
 & & &}$}\end{gathered} \right \} & \left \{ \begin{gathered}\scalebox{0.6}{$
\xymatrix @-1.0pc {
 & & & \\
& & & \circ \\
& & \circ  \ar@{-}[ur] & \\
\ar@{--}[dr] & \circ  \ar@{-}[ur] & \ar@{--}[dr] & \circ  \ar@{-}[ul] \\
 & \ar@{--}[ur] & \bullet \ar@{-}[ul] \ar@{-}[ur] & \\
 & & & \ar@{-}[ul] \bullet \\
 & & &}$
 $\xymatrix @-1.0pc {
  & & & \\
& & & \circ \\
& \ar@{--}[dr] & \circ  \ar@{-}[ur] & \\
\ar@{--}[ur] & \bullet  \ar@{-}[ur] & \ar@{--}[ur]& \bullet  \ar@{-}[ul] \\
 & & \bullet \ar@{-}[ul] \ar@{-}[ur] & \\
 & & & \ar@{-}[ul] \bullet \\
 & & &}$}\end{gathered} \right \}
 \vspace{10pt} \\
\left \{ 000111, 001110, 011100, 111000, 110001, 100011 \right \}
 &
 \left \{010101, 101010 \right \} \\

\end{array}$
\end{center}
\caption{The two orbits of $J(J([2] \times [2]))$ under $\pro$ (the dashed lines are the boundary paths corresponding to the order ideals) and the two orbits of binary words of length 6 of the form $w(1-w)$ under rotation (obtained from the boundary paths).}
\label{ex:halfsquare}
\end{figure}
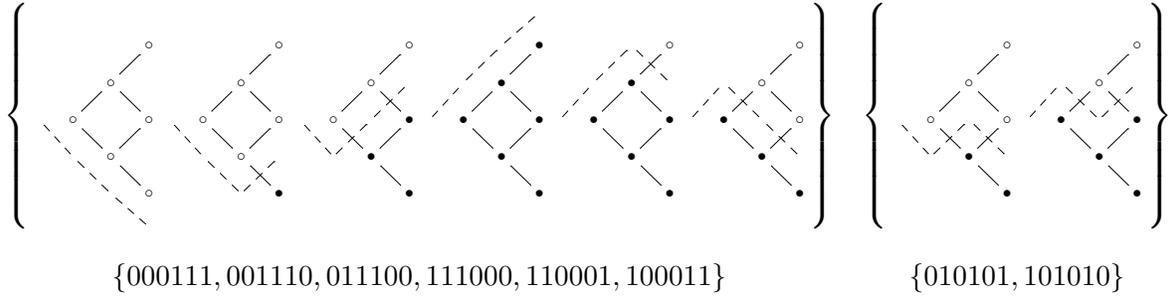

The set $X$ of binary words of the form $w (1-w)$, where $w$ is a binary word of length $n$, exhibits the CSP under rotation with the polynomial $\prod_{i=1}^n [2]_{q^i}$.

\begin{corollary}
Let $C_{2n}$ act on $J(J([2] \times [n-1]))$ by $\row$.  Then $\left(J(J([2] \times [n-1])),\prod_{i=1}^n [2]_{q^i},C_{2n}\right)$ exhibits the CSP.
\end{corollary}

We comment that the bijections in Sections~\ref{sec:nxk} and~\ref{sec:nxkb} are essentially the type $A$ and $B$ bijections given in~\cite{RushShiREU}, which benefited from an early draft of this paper to find a lovely uniform generalization of these two results to minuscule posets.

\subsection{\boldmath $\Phi^+(A_n)$}
\label{roottypea}

We remind the reader that we deal with the root posets of classical type case-by-case, losing the generality and uniformity of the main theorem in~\cite{Armstrong2011}.
Using D.\ White's equivariant bijection between $\mathcal{L}([2]\times[n+1])$ and noncrossing matchings in Theorem~\ref{thm:WhiteA}, we obtain the type $A_n$ case of Theorem~\ref{thm:armstrongetal}.

\begin{proof}[Proof of Theorem~\ref{thm:prorowcatalan}]
By the reasoning in Section~\ref{sec:prorowtoggle}, $\mathcal{L}([2]\times[n+1])$ under $\pro$ is in equivariant bijection with $J(\Phi^+(A_n))$ under $\pro$.  The result then follows from Theorem~\ref{thm:maintheorem}.
\end{proof}

An example is given in Figures~\ref{ex:syt33} and~\ref{ex:max33}.

\subsection{\boldmath$\Phi^+(B_n)$}
\label{roottypeb}

The type $B_n$ case of Theorem~\ref{thm:armstrongetal} also follows from a modification of the map in Theorem~\ref{prop:proback}, since $B_n$ noncrossing matchings are just the half-turn symmetric $A_{2n-1}$ matchings.

\begin{corollary}
There is an equivariant bijection between type $B_n$ noncrossing matchings under rotation and $J(\Phi^+ (B_n))$ under $\row$.
\end{corollary}

Figure~\ref{ex:b2noncross} illustrates this theorem for $n=2$.

\begin{figure}[ht]
\begin{center}
$\begin{array}{cc}

\left \{ \begin{gathered}\scalebox{0.7}{
$ \xymatrix @-1.2pc {
& & & & \\
& & & & \circ \\
& & \ar@{--}[dr] & \circ \ar@{-}[ur]  & \\
& \ar@{--}[ur] & \bullet \ar@{-}[ur]  & \ar@{--}[dr] & \circ \ar@{-}[ul] \\
\ar@{--}[ur] & & & & \\}$
$ \xymatrix @-1.2pc {
& & & & \\
& & & & \circ \\
& & & \circ \ar@{-}[ur]  & \\
& \ar@{--}[dr] & \circ \ar@{-}[ur]  &\ar@{--}[ur] & \bullet \ar@{-}[ul] \\
\ar@{--}[ur] & & \ar@{--}[ur] & & \\}$
$ \xymatrix @-1.2pc {
& & & & \\
& & & \ar@{--}[ur] & \bullet \\
& & \ar@{--}[ur] & \bullet \ar@{-}[ur]  & \\
& \ar@{--}[ur] & \bullet \ar@{-}[ur]  & & \bullet \ar@{-}[ul] \\
\ar@{--}[ur] & & & & \\}$
$ \xymatrix @-1.2pc {
& & & & \\
& & & \ar@{--}[dr] & \circ \\
& & \ar@{--}[ur] & \bullet \ar@{-}[ur]  & \\
& \ar@{--}[ur] & \bullet \ar@{-}[ur]  & & \bullet \ar@{-}[ul] \\
\ar@{--}[ur] & & & & \\}$
}\end{gathered} \right \}
&
\left \{ \begin{gathered}\scalebox{0.7}{
$ \xymatrix @-1.2pc {
& & & & \\
& & & & \circ \\
& & & \circ \ar@{-}[ur]  & \\
& \ar@{--}[dr] & \circ \ar@{-}[ur]  & \ar@{--}[dr] & \circ \ar@{-}[ul] \\
\ar@{--}[ur] & & \ar@{--}[ur] & & \\}$
$ \xymatrix @-1.2pc {
& & & & \\
& & & & \circ \\
& & \ar@{--}[dr] & \circ \ar@{-}[ur]  & \\
& \ar@{--}[ur] & \bullet \ar@{-}[ur]  & \ar@{--}[ur] & \bullet \ar@{-}[ul] \\
\ar@{--}[ur] & & & & \\}$}\end{gathered} \right \}
\vspace{10pt} \\
\scalebox{0.7}{
$\left \{ \quad \begin{gathered}
\setlength{\unitlength}{2mm}
\begin{picture}(12,12)
\bigcircle{5}{5}{5}

\put(5,10){\circle*{0.3}}
\put(8.53,8.53){\circle*{0.3}}
\put(10,5){\circle*{0.3}}
\put(8.53,1.46){\circle*{0.3}}
\put(5,0){\circle*{0.3}}
\put(1.46,1.46){\circle*{0.3}}
\put(0,5){\circle*{0.3}}
\put(1.46,8.53){\circle*{0.3}}

\qbezier[500](5,10)(5,5)(8.53,1.46)
\qbezier[500](10,5)(5,5)(8.53,8.53)
\qbezier[500](0,5)(5,5)(1.46,1.46)
\qbezier[500](5,0)(5,5)(1.46,8.53)

\put(4.7,10.3){\small{1}}
\put(8.7,8.7){\small{2}}
\put(10.3,4.7){\small{3}}
\put(8.7,0.6){\small{4}}
\put(4.7,-1){\small{-1}}
\put(-0.3,0.6){\small{-2}}
\put(-1.8,4.7){\small{-3}}
\put(-0.3,8.7){\small{-4}}

\end{picture} \qquad
\setlength{\unitlength}{2mm}
\begin{picture}(12,12)

\bigcircle{5}{5}{5}

\put(5,10){\circle*{0.3}}
\put(8.53,8.53){\circle*{0.3}}
\put(10,5){\circle*{0.3}}
\put(8.53,1.46){\circle*{0.3}}
\put(5,0){\circle*{0.3}}
\put(1.46,1.46){\circle*{0.3}}
\put(0,5){\circle*{0.3}}
\put(1.46,8.53){\circle*{0.3}}

\qbezier[500](5,10)(5,5)(8.53,8.53)
\qbezier[500](10,5)(5,5)(1.46,8.53)
\qbezier[500](5,0)(5,5)(1.46,1.46)
\qbezier[500](0,5)(5,5)(8.53,1.46)

\put(4.7,10.3){\small{1}}
\put(8.7,8.7){\small{2}}
\put(10.3,4.7){\small{3}}
\put(8.7,0.6){\small{4}}
\put(4.7,-1){\small{-1}}
\put(-0.3,0.6){\small{-2}}
\put(-1.8,4.7){\small{-3}}
\put(-0.3,8.7){\small{-4}}

\end{picture} \qquad
\setlength{\unitlength}{2mm}
\begin{picture}(12,12)

\bigcircle{5}{5}{5}

\put(5,10){\circle*{0.3}}
\put(8.53,8.53){\circle*{0.3}}
\put(10,5){\circle*{0.3}}
\put(8.53,1.46){\circle*{0.3}}
\put(5,0){\circle*{0.3}}
\put(1.46,1.46){\circle*{0.3}}
\put(0,5){\circle*{0.3}}
\put(1.46,8.53){\circle*{0.3}}

\qbezier[500](5,10)(5,5)(1.46,8.53)
\qbezier[500](8.53,8.53)(5,5)(0,5)
\qbezier[500](10,5)(5,5)(1.46,1.46)
\qbezier[500](5,0)(5,5)(8.53,1.46)

\put(4.7,10.3){\small{1}}
\put(8.7,8.7){\small{2}}
\put(10.3,4.7){\small{3}}
\put(8.7,0.6){\small{4}}
\put(4.7,-1){\small{-1}}
\put(-0.3,0.6){\small{-2}}
\put(-1.8,4.7){\small{-3}}
\put(-0.3,8.7){\small{-4}}

\end{picture} \qquad
\setlength{\unitlength}{2mm}
\begin{picture}(12,12)

\bigcircle{5}{5}{5}

\put(5,10){\circle*{0.3}}
\put(8.53,8.53){\circle*{0.3}}
\put(10,5){\circle*{0.3}}
\put(8.53,1.46){\circle*{0.3}}
\put(5,0){\circle*{0.3}}
\put(1.46,1.46){\circle*{0.3}}
\put(0,5){\circle*{0.3}}
\put(1.46,8.53){\circle*{0.3}}

\qbezier[500](5,10)(5,5)(1.46,1.46)
\qbezier[500](8.53,8.53)(5,5)(5,0)
\qbezier[500](10,5)(5,5)(8.53,1.46)
\qbezier[500](0,5)(5,5)(1.46,8.53)

\put(4.7,10.3){\small{1}}
\put(8.7,8.7){\small{2}}
\put(10.3,4.7){\small{3}}
\put(8.7,0.6){\small{4}}
\put(4.7,-1){\small{-1}}
\put(-0.3,0.6){\small{-2}}
\put(-1.8,4.7){\small{-3}}
\put(-0.3,8.7){\small{-4}}

\end{picture}
\end{gathered} \right \}$}
&

\scalebox{0.7}{
$\left \{ \quad \begin{gathered} \setlength{\unitlength}{2mm}
\begin{picture}(12,12)

\bigcircle{5}{5}{5}

\put(5,10){\circle*{0.3}}
\put(8.53,8.53){\circle*{0.3}}
\put(10,5){\circle*{0.3}}
\put(8.53,1.46){\circle*{0.3}}
\put(5,0){\circle*{0.3}}
\put(1.46,1.46){\circle*{0.3}}
\put(0,5){\circle*{0.3}}
\put(1.46,8.53){\circle*{0.3}}

\qbezier[500](5,10)(5,5)(8.53,8.53)
\qbezier[500](0,5)(5,5)(1.46,8.53)
\qbezier[500](5,0)(5,5)(1.46,1.46)
\qbezier[500](10,5)(5,5)(8.53,1.46)

\put(4.7,10.3){\small{1}}
\put(8.7,8.7){\small{2}}
\put(10.3,4.7){\small{3}}
\put(8.7,0.6){\small{4}}
\put(4.7,-1){\small{-1}}
\put(-0.3,0.6){\small{-2}}
\put(-1.8,4.7){\small{-3}}
\put(-0.3,8.7){\small{-4}}

\end{picture} \qquad
\setlength{\unitlength}{2mm}
\begin{picture}(12,12)

\bigcircle{5}{5}{5}

\put(5,10){\circle*{0.3}}
\put(8.53,8.53){\circle*{0.3}}
\put(10,5){\circle*{0.3}}
\put(8.53,1.46){\circle*{0.3}}
\put(5,0){\circle*{0.3}}
\put(1.46,1.46){\circle*{0.3}}
\put(0,5){\circle*{0.3}}
\put(1.46,8.53){\circle*{0.3}}

\qbezier[500](10,5)(5,5)(8.53,8.53)
\qbezier[500](5,10)(5,5)(1.46,8.53)
\qbezier[500](0,5)(5,5)(1.46,1.46)
\qbezier[500](5,0)(5,5)(8.53,1.46)

\put(4.7,10.3){\small{1}}
\put(8.7,8.7){\small{2}}
\put(10.3,4.7){\small{3}}
\put(8.7,0.6){\small{4}}
\put(4.7,-1){\small{-1}}
\put(-0.3,0.6){\small{-2}}
\put(-1.8,4.7){\small{-3}}
\put(-0.3,8.7){\small{-4}}

\end{picture} \end{gathered} \right \}$}
\\
\end{array}$
\end{center}
\caption{The two orbits of $J(\Phi^+(B_2))$ under $\pro$ (the dashed lines are the boundary paths corresponding to the order ideals) and the two orbits of type $B_2$ noncrossing matchings under rotation (obtained from the boundary path by taking $i$ to be the smaller element of its block if the $i$th step was northeast).}
\label{ex:b2noncross}
\end{figure}
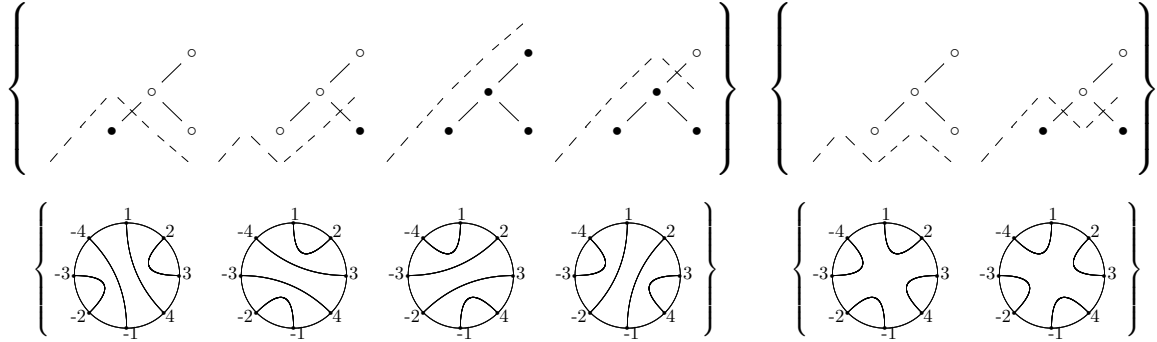

\section{RC-Posets of Height Greater than $1$}
\label{sec:threedim}

In this section, we apply Theorem~\ref{thm:maintheorem} to the following \rcposet{}s of height greater than one:  $\Phi^+(D_n)$ and $[\ell]\times[m]\times[n]$. In the case of $[2]\times[m]\times[n]$, we obtain a bijection which turns $\row$ into a rotation on noncrossing partitions of $[n+m+1]$ into $m+1$ blocks. 

\subsection{\boldmath $\Phi^+(D_n)$}

The poset $\Phi^+(D_n)$ is a copy of $\Phi^+(A_{n-1})$ joined with $J([2] \times [n-2])$ (see Figure~\ref{ex:rootposets}).  We choose to draw $\Phi^+(D_n)$ as an \rcposet{} of height 2 by letting the elements $e_i-e_n$ and $e_i+e_n$ occupy the same positions.  Then---ignoring elements and edges between elements in same position---$\Phi^+(D_n)$ looks exactly like $\Phi^+(B_n)$.   To define the position map, we let $\pi(e_i-e_j)=(i+j,j-i)$ and $\pi(e_i+e_j)=(2n-(j-i),2n-(i+j))$ for $i<j$.  For example, $D_4$ is drawn this way in Figure~\ref{ex:d4rc}.

\begin{figure}[ht]
\begin{center}
\includegraphics[scale=0.5]{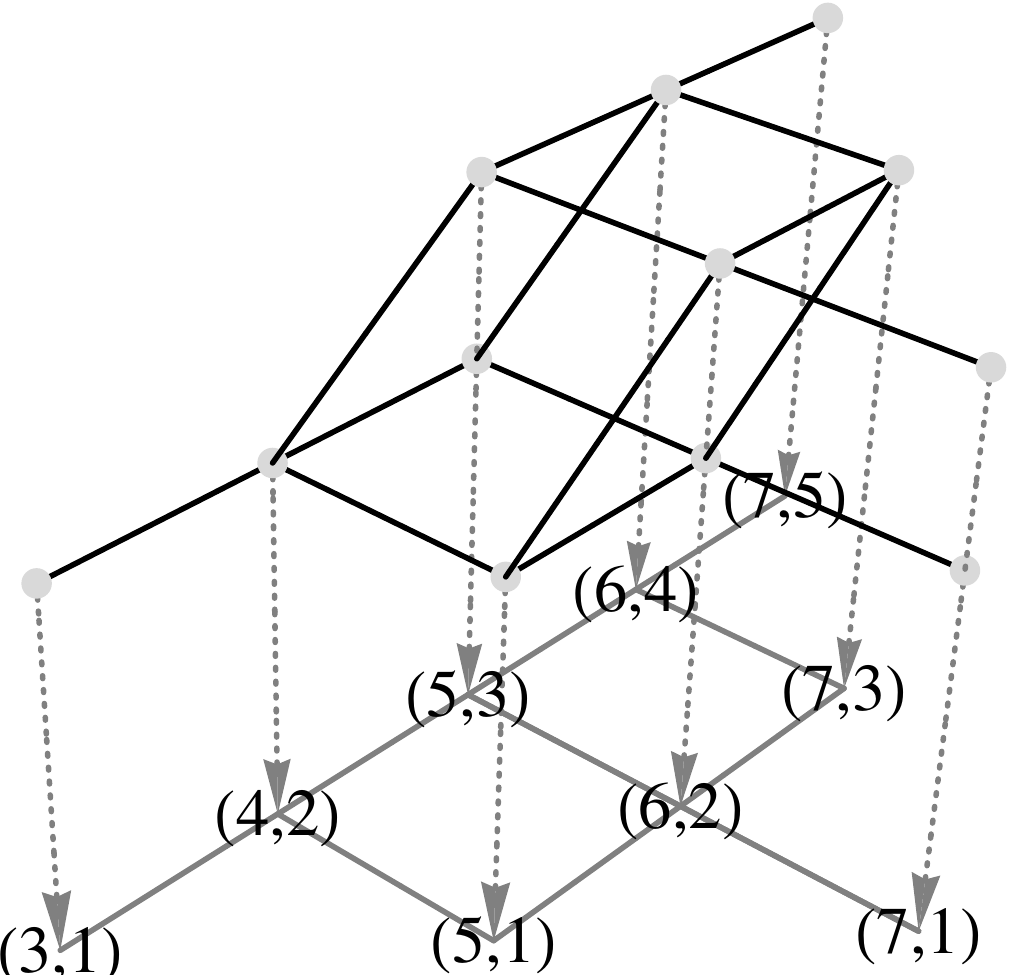}
\end{center}
\caption{$\Phi^+(D_4)$ drawn as an \rcposet{} of height 2.}
\label{ex:d4rc}
\end{figure}

As a corollary of Theorem~\ref{thm:maintheorem} we obtain the following.

\begin{corollary}
\label{thm:typeD}
There is an equivariant bijection between $J(\Phi^+(D_n))$ under $\row$ and $J(\Phi^+(D_n))$ under $\pro$.
\end{corollary}

Recall that type $D_n$ noncrossing matchings are defined to be half-turn symmetric noncrossing matchings satisfying a certain parity condition on $4n-4$ points around a large circle and $4$ points around a smaller interior circle~\cite{athanasiadis2005noncrossing}.  Rotation on these is defined by rotating the inner and outer circles in opposite directions.

Let $I \in J(\Phi^+(D_n))$ have the property that $e_i + e_n \in J$ if and only if $e_i - e_n \in J$.  Then $\pro(I)$ also has this property, so that the
entire orbit is mirrored by the corresponding orbit in $J(\Phi^+(B_n))$ obtained by identifying $e_i + e_n$ and $e_i - e_n$.  Thus, such order ideals are in bijection with those type $D_n$ noncrossing matchings that have no matchings between the outer vertices and the four inner vertices.  The general case was solved in~\cite{Armstrong2011}, and the second author has found a new approach that will appear in a subsequent paper.

Conjugating promotion to rowmotion, we comment that our bijections from Sections~\ref{roottypea} and~\ref{roottypeb} for types $A$ and $B$ and our partial result in this section for type $D$ must then be the same as those found in~\cite{Armstrong2011} by the uniqueness result in Section 5.4 of that paper.

\subsection{Plane Partitions}
\label{sec:ppsec}

In this section, consider the order ideals of the product of three chains---that is, plane partitions---under rowmotion.  We draw $[\ell]\times[m]\times[n]$ as an \rcposet{} of height $\ell$ to generalize the approach in Theorem~\ref{prop:proback}. When $\ell=2$, we prove Theorem~\ref{thm:ncpp}, which gives an equivariant bijection between $J([2]\times[m]\times[n])$ under rowmotion and noncrossing partitions of $[n+m+1]$ into $m+1$ blocks under rotation, simplifying proofs of Theorem~\ref{thm:pporder} due to P.\ Cameron and D.\ Fon-der-Flaass~\cite{cameron1995orbits}, as well as Corollary~\ref{cor:vicconj} of D.\ B Rush and X.\ Shi~\cite{RushShiREU}.

We interpret $[\ell]\times[m]\times[n]$ as an \rcposet{} by drawing each layer $\{i\}\times[m]\times[n]$ for $1\leq i \leq \ell$ as an \rcposet{} and then letting $\pi(i,j,k)=(i-j+k,i+j+k)$.  For example, see Figure~\ref{ex:ppex}.

\begin{figure}[ht]
\begin{center}
$\begin{array}{cc}

\begin{gathered}\includegraphics[scale=0.5]{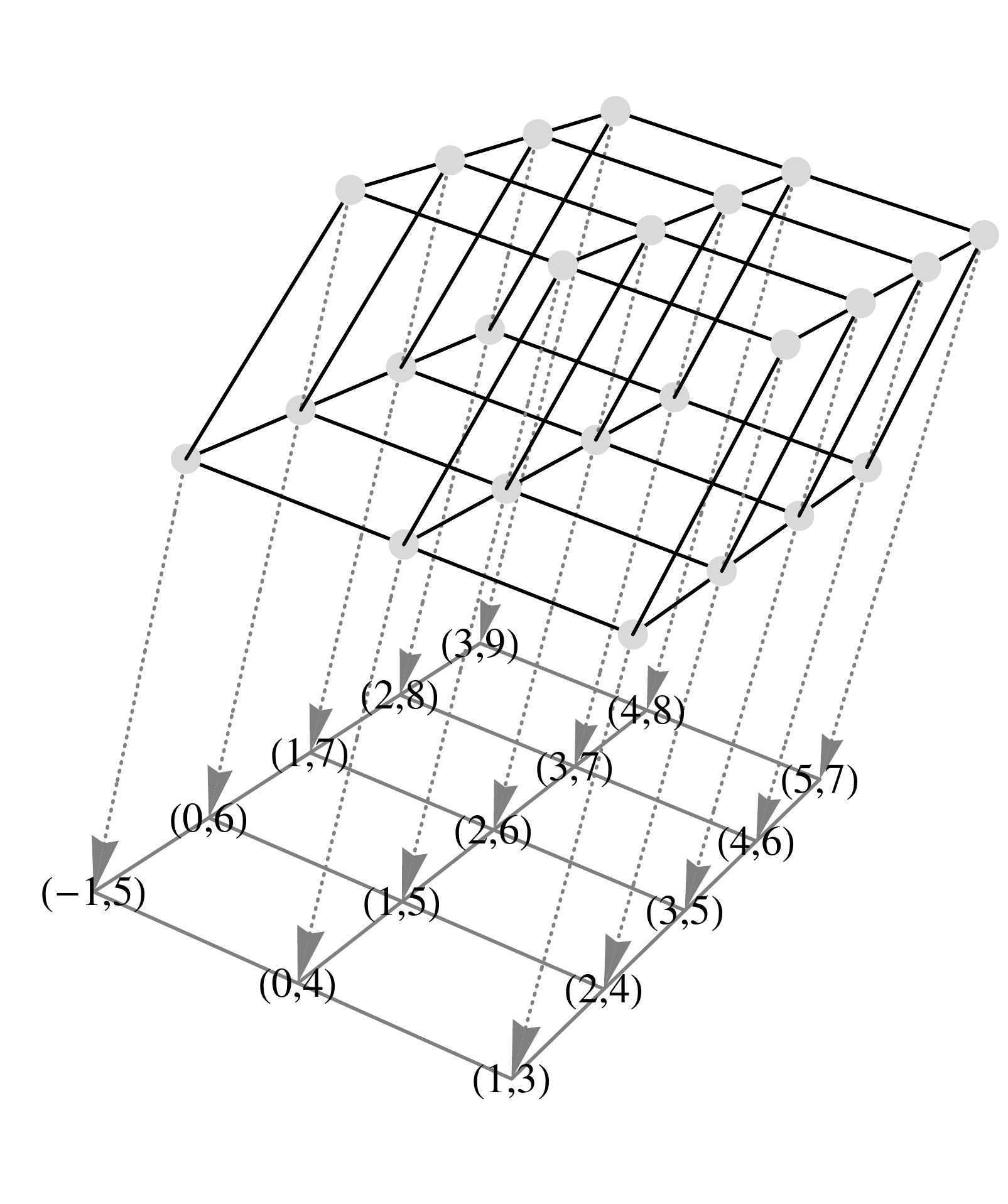}\end{gathered} &
\begin{gathered}\includegraphics[scale=0.5]{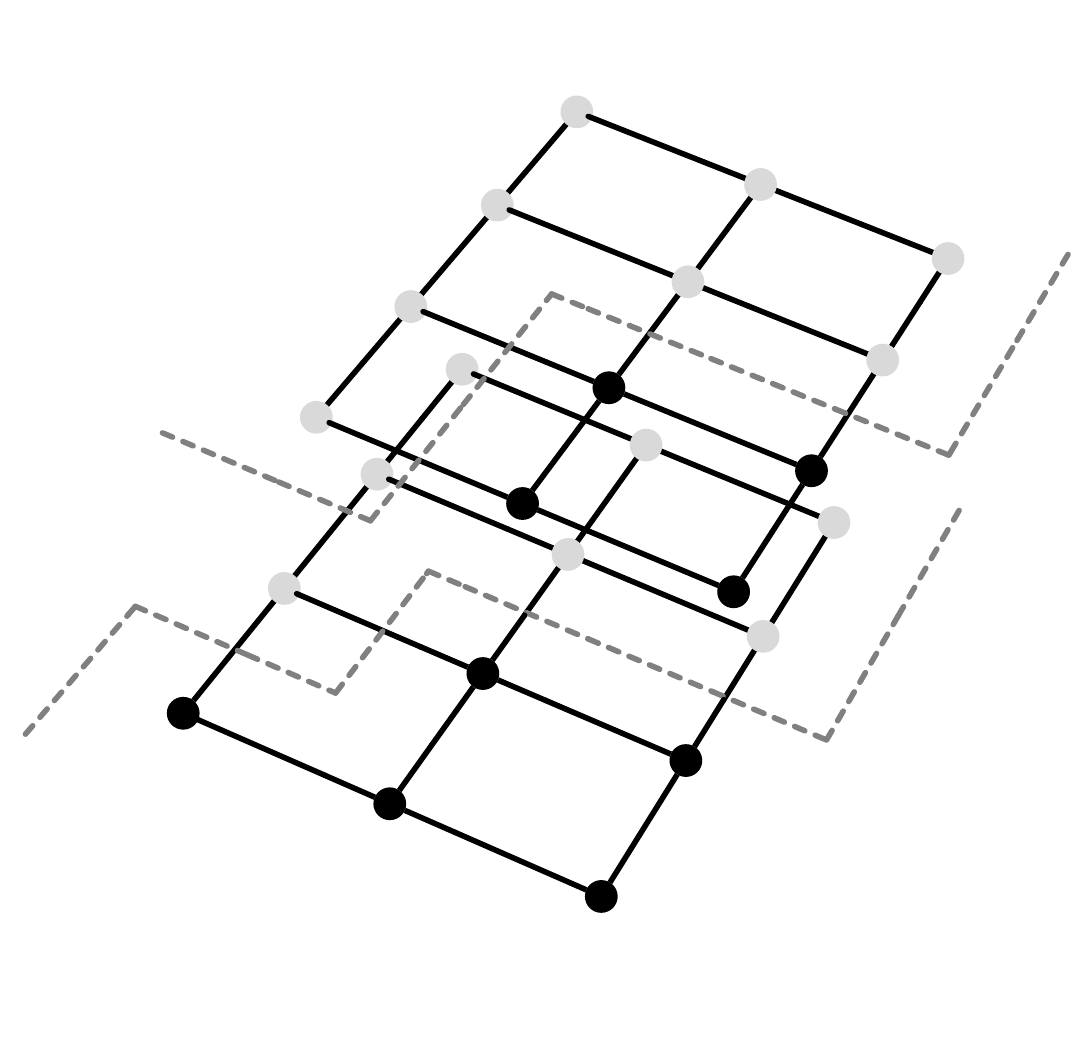}\end{gathered} \\
\end{array}$
\end{center}

\caption{On the left is $[2]\times[3]\times[4]$ drawn as an \rcposet{} of height 2.  When there are two elements with the same position, the second element is raised; the position is indicated by a dotted arrow down.  Covering relations are drawn with solid black lines, and are projected down as solid gray lines.  On the right are the order ideal and boundary paths corresponding to the rightmost plane partition in
Figure~\ref{ex:pppro} (covering relations between layers are suppressed).}
\label{ex:ppex}
\end{figure}

As usual, we immediately obtain the following corollary of Theorem~\ref{thm:maintheorem}.

\begin{corollary}
\label{thm:pp}
There is an equivariant bijection between $J([\ell]\times[m]\times[n])$ under
$\row$ and $J([\ell]\times[m]\times[n])$ under $\pro$.
\end{corollary}

Figure~\ref{ex:pppro} displays an orbit of $J([2]\times[3]\times[4])$ under promotion (drawn
using code written by J.\ S.\ Kim for Ti\textit{k}Z).

\begin{figure}[ht]
\begin{center}
$\left \{ \begin{gathered}\scalebox{0.35}{
\begin{tikzpicture}
\planepartition{{2,1,1,1},{2,1}}
\end{tikzpicture}
\begin{tikzpicture}
\planepartition{{2,2,2,2},{2,1,1,1},{2,1}}
\end{tikzpicture}
\begin{tikzpicture}
\planepartition{{2,2,2,1},{1,1,1},{1}}
\end{tikzpicture}
\begin{tikzpicture}
\planepartition{{2,2,1},{2,2}}
\end{tikzpicture}
\begin{tikzpicture}
\planepartition{{2,2,2,2},{2,2,1},{2,2}}
\end{tikzpicture}
\begin{tikzpicture}
\planepartition{{2,2,2,1},{2,1,1,1},{2,1,1}}
\end{tikzpicture}
\begin{tikzpicture}
\planepartition{{2,2,1},{1,1,1},{1,1}}
\end{tikzpicture}
\begin{tikzpicture}
\planepartition{{2,2},{2,2},{1}}
\end{tikzpicture}
}\end{gathered} \right \}.$
\end{center}
\caption{An orbit of $J([2]\times[3]\times[4])$ under promotion.}
\label{ex:pppro}
\end{figure}
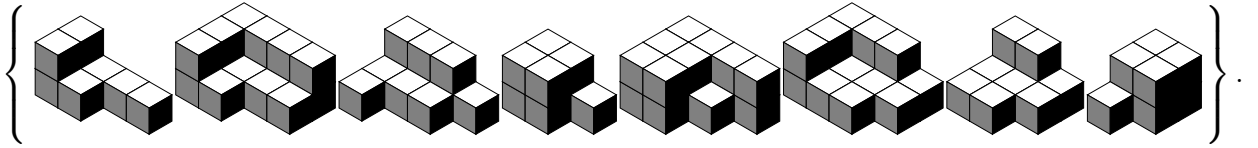

We now extend the boundary paths of Definition~\ref{def:boundarypath} to $J([\ell]\times[m]\times[n])$.

\begin{definition} \rm
\label{def:boundmx}
Let $I\in J([\ell]\times[m]\times[n])$.  We define its \emph{boundary path matrix} $B(I)=\{b_{i,j}\}$ to be the $\ell\times (m+n+\ell-1)$ matrix with row $i$ containing the boundary path of layer $\{i\}\times[m]\times[n]$, preceded by $i-1$ zeros and succeeded by $\ell-i$ zeros.
\end{definition}

Note that the rows of the boundary path matrix each sum to $n$.  Because of the covering relations of $[\ell]\times[m]\times[n]$, they also satisfy the condition:
\[
\mbox{if } \sum_{j=1}^k b_{i,j} = \sum_{j=1}^k b_{i+1,j}, \mbox{ then }
b_{i+1,j+1} \neq 1.
\]

From our characterization of $\pro$ as an action on the columns of an \rcposet{}, it is clear that promotion traces from left to right through the columns of the boundary path matrix, swapping each pair of entries in adjacent columns and the same row that result in a matrix still satisfying the condition above.  Figure~\ref{ex:ppprob} translates the order ideals of Figure~\ref{ex:pppro} to boundary path matrices.

\begin{figure}[ht]
\begin{center}
$\left \{ \begin{gathered} \scalebox{0.8}{$\left(\begin{array}{cccccccc}
0&1&1&0&1&1&0&0\\
0&0&1&0&0&1&1&1
\end{array}\right)
\left(\begin{array}{cccccccc}
1&1&0&1&1&0&0&0\\
0&1&0&0&1&1&1&0
\end{array}\right)
\left(\begin{array}{cccccccc}
1&0&1&1&0&1&0&0\\
0&0&0&1&1&1&0&1
\end{array}\right)
\left(\begin{array}{cccccccc}
0&1&1&0&1&0&1&0\\
0&0&1&1&0&0&1&1
\end{array}\right) $}\end{gathered} \right.$
\end{center}

\begin{center}
$\left. \begin{gathered} \scalebox{0.8}{$
\left(\begin{array}{cccccccc}
1&1&0&1&0&1&0&0\\
0&1&1&0&0&1&1&0
\end{array}\right)
\left(\begin{array}{cccccccc}
1&1&1&0&1&0&0&0\\
0&1&0&0&1&1&0&1
\end{array}\right)
\left(\begin{array}{cccccccc}
1&1&0&1&0&0&1&0\\
0&0&0&1&1&0&1&1
\end{array}\right)
\left(\begin{array}{cccccccc}
1&0&1&0&0&1&1&0\\
0&0&1&1&0&0&1&1
\end{array}\right)$} \end{gathered} \right \} .$
\end{center}
\caption{The boundary path matrices of the order ideals in
Figure~\ref{ex:pppro}.}
\label{ex:ppprob}
\end{figure}

Using boundary path matrices, we can easily determine a factor of the order of rowmotion on plane partitions of general height.

\begin{theorem}
\label{thm:genpporder}
$J([\ell]\times[m]\times[n])$ under $\row$ has order divisible by $m+n+\ell-1$.
\end{theorem}

\begin{proof}
By Corollary~\ref{thm:pp}, $\row$ and $\pro$ have the same order.  We show that $\pro^{m+n+\ell-1}$ applied to the empty order ideal is the identity.

From left to right, the $i$th row of the boundary path matrix of the empty order ideal has $m+i-1$ zeros, $n$ ones, and then $\ell-i$ zeros.  From the definition, applying $\pro^m$ to the empty order ideal cycles the first $m$ all-zero columns to the end.  This corresponds to the full order ideal containing all elements of $[\ell]\times[m]\times[n]$.  Each subsequent application of $\pro$ fixes every nonzero column except the one directly to the left of the first all-zero column.  Promotion then sends this column to the right of the all-zero columns.  Thus, $\pro^{n+\ell-1}$ takes the full order ideal back to the empty order ideal.
\end{proof}

In~\cite{cameron1995orbits}, P.\ Cameron and D.\ Fon-der-Flaass proved that when $m+n+\ell-1=p$ is prime with $n\geq (\ell-1)(m-1)$, $p$ divides the length of every orbit of $J([\ell]\times[m]\times[n])$ under rowmotion.  They further conjectured that the restriction $n\geq (\ell-1)(m-1)$ could be removed.
Having already dealt with the $[1]\times[m]\times[n]$ case in 
Section~\ref{sec:nxk},
we now examine more closely the case $\ell=2$.  
In the same paper,
P.\ Cameron and D.\ Fon-der-Flaass prove the following theorem
for $[2]\times[m]\times[n]$.

\begin{theorem}[P.\ Cameron, D.\ Fon-der-Flaass]
\label{thm:pporder}
	The order of $\row$ on $J([2]\times[m]\times[n])$ is $m+n+1$.
\end{theorem}

They proved this theorem by constructing a bijection between entire orbits of $J([2]\times[m]\times[n])$ under a conjugate of rowmotion and orbits of certain words under an action $\psi$ whose order is more easily analyzed.  By interpreting plane partitions as \rcposet{}s and encoding them as boundary path matrices, we will simplify this proof by giving a bijection from plane partitions under promotion to these words under $\psi$.   Finally, we give a simple equivariant bijection from the words under $\psi$ to noncrossing partitions of $[n+m+1]$ under rotation, from which is easy to see the order must be $m+n+1$.

A word containing parentheses is called \textit{balanced} if the number of left parentheses is always greater than or equal to the number of right parentheses.

\begin{definition}[P.\ Cameron, D.\ Fon-der-Flaass] \rm
\label{def:bracketwords}
Let $\beta_{m,n}$ be all balanced words of length $m+n+1$ on the alphabet $\left\{(,),\bullet,\fbox{$)($}\right\}$ with $m$ left parentheses and $m$ right parentheses (including those in a \fbox{$)($} symbol).

Define an action $\psi$ on $\beta_{m,n}$ as follows:
\begin{enumerate}
	\item $\psi\left[\bullet A_1 \right] = A_1 \bullet$,
	\item $\psi\left[(A_1) A_2\right] = A_1(A_2)$,
	\item $\psi\left[(A_1 \fbox{$)($} A_2 \fbox{$)($} \ldots \fbox{$)($} A_k
)A_{k+1} \right] = A_1 (A_2 \fbox{$)($} \ldots \fbox{$)($} A_{k} \fbox{$)($}
A_{k+1})$,
\end{enumerate}
where each of $A_1,A_2,\ldots,A_{k+1}$ is a balanced subword.
\end{definition}

\begin{theorem}[P.\ Cameron, D.\ Fon-der-Flaass]
\label{thm:bracket}
	There is an equivariant bijection between $J([2]\times[m]\times[n])$ under $\row$ and $\beta_{m,n}$ under $\psi$.
\end{theorem}

Recall that the generating function for the number of boxes contained in plane partitions inside an $[\ell]\times[m]\times[n]$ box is given by a $q$-ification of the MacMahon box formula~\cite{Macmahon}.

\[M_{\ell,m,n}(q)= \prod_{1\leq i\leq \ell, 1\leq j\leq m, 1\leq k\leq n} \frac{[i+j+k-1]_q}{[i+j+k-2]_q} \]

In general, $(J([\ell]\times[m]\times[n]), M_{\ell,m,n}(q), C_{\ell+m+n-1})$ does not exhibit the CSP, where the cyclic group acts by rowmotion.  For $[3]\times[3]\times[3]$, one can check that the polynomial fails to exhibit the CSP, while for $[4]\times[4]\times[4]$, K.\ Dilks computed that there exist orbits of size $33=3 \cdot (4+4+4-1)$.

V.\ Reiner conjectured that $(J([2]\times[m]\times[n]),M_{2,m,n}(q),C_{m+n+1})$ exhibited the CSP.  D.\ B Rush and X.\ Shi recently proved this using P.\ Cameron and D.\ Fon-der-Flaass's equivariant bijection to $\beta_{m,n}$~\cite{RushShiREU}.  Their theorem, which we obtain as a corollary, was the inspiration for our bijection to noncrossing partitions in Theorem~\ref{thm:ncpp}.

\begin{theorem}
\label{thm:ncpp}
	There is an equivariant bijection between $J([2]\times[m]\times[n])$ under $\row$ and noncrossing partitions of $[n+m+1]$ into $m+1$ blocks under rotation.
\end{theorem}

\begin{proof}
By Corollary~\ref{thm:pp}, we may consider $J([2]\times[m]\times[n])$ under $\pro$.

We first convert a boundary path matrix to a balanced word in $\beta_{m,n}$
using the correspondence below on columns of the boundary path matrix.

\[\left(\begin{array}{c}
1\\
0
\end{array}\right) \leftrightarrow \mbox{`$($'} \qquad
\left(\begin{array}{c}
0\\
1
\end{array}\right) \leftrightarrow \mbox{`$)$'} \qquad
\left(\begin{array}{c}
1\\
1
\end{array}\right) \leftrightarrow \fbox{$)($} \qquad
\left(\begin{array}{c}
0\\
0
\end{array}\right) \leftrightarrow \bullet\]

Note that the boundary path matrix condition given after Definition~\ref{def:boundmx} is equivalent to saying that the resulting words are balanced.  Figure~\ref{ex:ppprowords} translates the boundary path matrices of Figure~\ref{ex:ppprob} to balanced words.

We show that this bijection is equivariant, using the definition of $\psi$ in Definition~\ref{def:bracketwords}.

The first rule, $\psi\left[\bullet A\right] = A \bullet$, corresponds to the case when the first column of the boundary path matrix is $(0,0)^\intercal$.  This column can swap with all other columns without violating the boundary path matrix condition, and so it is moved to the end of the word under promotion.

Consider when the first column is $(1,0)^\intercal$.  This column can swap with $(0,0)^\intercal$ and $(1,0)^\intercal$ without violating the boundary path matrix condition, but it cannot swap with $(0,1)^\intercal$ or $(1,1)^\intercal$.

The second rule, $\psi\left[(A_1) A_2\right] = A_1(A_2)$, corresponds to when the first column is $(1,0)^\intercal$ and the first column it encounters that it cannot swap with is $(0,1)^\intercal$.  In this case, the $(1,0)^\intercal$ remains fixed and the $(0,1)^\intercal$ is free to move to the end of the word.

The third rule, $\psi\left[(A_1 \fbox{$)($} A_2 \fbox{$)($} \ldots \fbox{$)($} A_{k} )A_{k+1} \right] = A_1 (A_2 \fbox{$)($} \ldots \fbox{$)($} A_{k} \fbox{$)($} A_{k+1})$, corresponds to when the first column $(1,0)^\intercal$ encounters $(1,1)^\intercal$ first.  Then the $(1,0)^\intercal$ remains and the $(1,1)^\intercal$ can swap to the right without violating the boundary path matrix condition until it reaches the first $(0,1)^\intercal$ such that the columns to the left have the same number of $1$s in the top and bottom rows.  This $(0,1)^\intercal$ then continues to the end of the word.

We now give an equivariant bijection from $\beta_{m,n}$ under $\psi$ to noncrossing partitions of $[n+m+1]$ into $m+1$ blocks under rotation.  For $i<j$, if `$($' in position $i$ is paired with `$)$' in position $j$---including brackets from the symbol \fbox{$)($}---then $i$ and $j$ are in a block together.  The resulting noncrossing partition will have exactly $m+1$ blocks because there are $m+1$ $0$'s in the bottom row of the boundary path matrix: each $(0,0)^\intercal$ column is replaced by a $\bullet$, which corresponds to a singleton block, and each $(1,0)^\intercal$ column becomes a `$($', which corresponds to the first element in a block.  For an example, see Figure~\ref{ex:ppprowords}.  It is clear that this bijection is equivariant.
\end{proof}

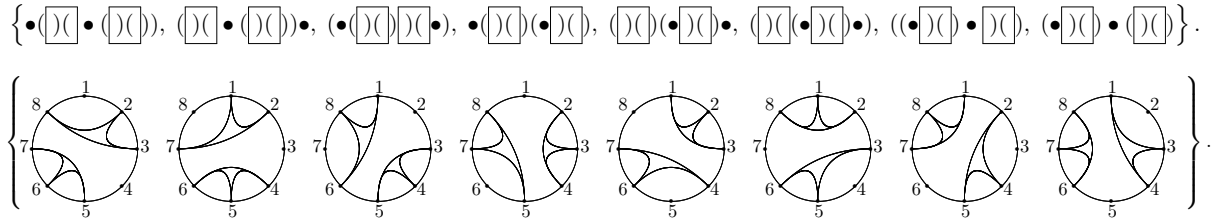
\begin{figure}[ht]
\begin{center}

\scalebox{0.8}{$\left \{ \bullet ( \fbox{$)($} \bullet ( \fbox{$)($} ) ), \;
( \fbox{$)($} \bullet ( \fbox{$)($} ) ) \bullet, \;
( \bullet ( \fbox{$)($} ) \fbox{$)($} \bullet ), \;
\bullet ( \fbox{$)($} ) ( \bullet \fbox{$)($} ), \; ( \fbox{$)($} ) ( \bullet \fbox{$)($} ) \bullet, \;
( \fbox{$)($} ( \bullet \fbox{$)($} ) \bullet ), \;
( ( \bullet \fbox{$)($} ) \bullet \fbox{$)($} ), \;
( \bullet \fbox{$)($} ) \bullet ( \fbox{$)($} ) \right \}.$}

\vspace{10pt}

\scalebox{0.7}{
$\left \{ \; \begin{gathered} \setlength{\unitlength}{2mm}
\begin{picture}(12,12)

\bigcircle{5}{5}{5}

\put(5,10){\circle*{0.3}}
\put(8.53,8.53){\circle*{0.3}}
\put(10,5){\circle*{0.3}}
\put(8.53,1.46){\circle*{0.3}}
\put(5,0){\circle*{0.3}}
\put(1.46,1.46){\circle*{0.3}}
\put(0,5){\circle*{0.3}}
\put(1.46,8.53){\circle*{0.3}}

\qbezier[500](8.53,8.53)(5,5)(10,5)
\qbezier[500](10,5)(5,5)(1.46,8.53)
\qbezier[500](5,0)(5,5)(1.46,1.46)
\qbezier[500](1.46,1.46)(5,5)(0,5)
\qbezier[500](8.53,8.53)(5,5)(1.46,8.53)
\qbezier[500](5,0)(5,5)(0,5)

\put(4.7,10.3){\small{1}}
\put(8.7,8.7){\small{2}}
\put(10.3,4.7){\small{3}}
\put(8.7,0.6){\small{4}}
\put(4.7,-1.5){\small{5}}
\put(0,0.6){\small{6}}
\put(-1.2,4.7){\small{7}}
\put(0,8.7){\small{8}}

\end{picture} \quad
\setlength{\unitlength}{2mm}
\begin{picture}(12,12)

\bigcircle{5}{5}{5}

\put(5,10){\circle*{0.3}}
\put(8.53,8.53){\circle*{0.3}}
\put(10,5){\circle*{0.3}}
\put(8.53,1.46){\circle*{0.3}}
\put(5,0){\circle*{0.3}}
\put(1.46,1.46){\circle*{0.3}}
\put(0,5){\circle*{0.3}}
\put(1.46,8.53){\circle*{0.3}}

\qbezier[500](5,10)(5,5)(8.53,8.53)
\qbezier[500](8.53,8.53)(5,5)(0,5)
\qbezier[500](8.53,1.46)(5,5)(5,0)
\qbezier[500](5,0)(5,5)(1.46,1.46)
\qbezier[500](5,10)(5,5)(0,5)
\qbezier[500](8.53,1.46)(5,5)(1.46,1.46)

\put(4.7,10.3){\small{1}}
\put(8.7,8.7){\small{2}}
\put(10.3,4.7){\small{3}}
\put(8.7,0.6){\small{4}}
\put(4.7,-1.5){\small{5}}
\put(0,0.6){\small{6}}
\put(-1.2,4.7){\small{7}}
\put(0,8.7){\small{8}}

\end{picture} \quad
\setlength{\unitlength}{2mm}
\begin{picture}(12,12)


\bigcircle{5}{5}{5}

\put(5,10){\circle*{0.3}}
\put(8.53,8.53){\circle*{0.3}}
\put(10,5){\circle*{0.3}}
\put(8.53,1.46){\circle*{0.3}}
\put(5,0){\circle*{0.3}}
\put(1.46,1.46){\circle*{0.3}}
\put(0,5){\circle*{0.3}}
\put(1.46,8.53){\circle*{0.3}}

\qbezier[500](1.46,8.53)(5,5)(5,10)
\qbezier[500](5,10)(5,5)(1.46,1.46)
\qbezier[500](10,5)(5,5)(8.53,1.46)
\qbezier[500](8.53,1.46)(5,5)(5,0)
\qbezier[500](1.46,8.53)(5,5)(1.46,1.46)
\qbezier[500](10,5)(5,5)(5,0)

\put(4.7,10.3){\small{1}}
\put(8.7,8.7){\small{2}}
\put(10.3,4.7){\small{3}}
\put(8.7,0.6){\small{4}}
\put(4.7,-1.5){\small{5}}
\put(0,0.6){\small{6}}
\put(-1.2,4.7){\small{7}}
\put(0,8.7){\small{8}}

\end{picture} \quad
\setlength{\unitlength}{2mm}
\begin{picture}(12,12)


\bigcircle{5}{5}{5}

\put(5,10){\circle*{0.3}}
\put(8.53,8.53){\circle*{0.3}}
\put(10,5){\circle*{0.3}}
\put(8.53,1.46){\circle*{0.3}}
\put(5,0){\circle*{0.3}}
\put(1.46,1.46){\circle*{0.3}}
\put(0,5){\circle*{0.3}}
\put(1.46,8.53){\circle*{0.3}}

\qbezier[500](1.46,8.53)(5,5)(5,0)
\qbezier[500](5,0)(5,5)(0,5)
\qbezier[500](10,5)(5,5)(8.53,1.46)
\qbezier[500](10,5)(5,5)(8.53,8.53)
\qbezier[500](8.53,1.46)(5,5)(8.53,8.53)
\qbezier[500](1.46,8.53)(5,5)(0,5)

\put(4.7,10.3){\small{1}}
\put(8.7,8.7){\small{2}}
\put(10.3,4.7){\small{3}}
\put(8.7,0.6){\small{4}}
\put(4.7,-1.5){\small{5}}
\put(0,0.6){\small{6}}
\put(-1.2,4.7){\small{7}}
\put(0,8.7){\small{8}}

\end{picture} \quad
\setlength{\unitlength}{2mm}
\begin{picture}(12,12)


\bigcircle{5}{5}{5}

\put(5,10){\circle*{0.3}}
\put(8.53,8.53){\circle*{0.3}}
\put(10,5){\circle*{0.3}}
\put(8.53,1.46){\circle*{0.3}}
\put(5,0){\circle*{0.3}}
\put(1.46,1.46){\circle*{0.3}}
\put(0,5){\circle*{0.3}}
\put(1.46,8.53){\circle*{0.3}}

\qbezier[500](8.53,8.53)(5,5)(5,10)
\qbezier[500](8.53,1.46)(5,5)(1.46,1.46)
\qbezier[500](10,5)(5,5)(5,10)
\qbezier[500](8.53,1.46)(5,5)(0,5)
\qbezier[500](0,5)(5,5)(1.46,1.46)
\qbezier[500](10,5)(5,5)(8.53,8.53)

\put(4.7,10.3){\small{1}}
\put(8.7,8.7){\small{2}} 
\put(10.3,4.7){\small{3}}
\put(8.7,0.6){\small{4}}
\put(4.7,-1.5){\small{5}}
\put(0,0.6){\small{6}}
\put(-1.2,4.7){\small{7}} 
\put(0,8.7){\small{8}}

\end{picture} \quad
\setlength{\unitlength}{2mm}
\begin{picture}(12,12)


\bigcircle{5}{5}{5}

\put(5,10){\circle*{0.3}}
\put(8.53,8.53){\circle*{0.3}}
\put(10,5){\circle*{0.3}}
\put(8.53,1.46){\circle*{0.3}}
\put(5,0){\circle*{0.3}}
\put(1.46,1.46){\circle*{0.3}}
\put(0,5){\circle*{0.3}}
\put(1.46,8.53){\circle*{0.3}}

\qbezier[500](8.53,8.53)(5,5)(5,10)
\qbezier[500](5,0)(5,5)(1.46,1.46)
\qbezier[500](1.46,8.53)(5,5)(5,10)
\qbezier[500](5,0)(5,5)(10,5)
\qbezier[500](10,5)(5,5)(1.46,1.46)
\qbezier[500](1.46,8.53)(5,5)(8.53,8.53)

\put(4.7,10.3){\small{1}}
\put(8.7,8.7){\small{2}} 
\put(10.3,4.7){\small{3}}
\put(8.7,0.6){\small{4}}
\put(4.7,-1.5){\small{5}}
\put(0,0.6){\small{6}}
\put(-1.2,4.7){\small{7}} 
\put(0,8.7){\small{8}}

\end{picture} \quad
\setlength{\unitlength}{2mm}
\begin{picture}(12,12)


\bigcircle{5}{5}{5}

\put(5,10){\circle*{0.3}}
\put(8.53,8.53){\circle*{0.3}}
\put(10,5){\circle*{0.3}}
\put(8.53,1.46){\circle*{0.3}}
\put(5,0){\circle*{0.3}}
\put(1.46,1.46){\circle*{0.3}}
\put(0,5){\circle*{0.3}}
\put(1.46,8.53){\circle*{0.3}}

\qbezier[500](0,5)(5,5)(5,10)
\qbezier[500](5,0)(5,5)(8.53,1.46)
\qbezier[500](1.46,8.53)(5,5)(5,10)
\qbezier[500](5,0)(5,5)(8.53,8.53)
\qbezier[500](8.53,1.46)(5,5)(8.53,8.53)
\qbezier[500](1.46,8.53)(5,5)(0,5)

\put(4.7,10.3){\small{1}}
\put(8.7,8.7){\small{2}} 
\put(10.3,4.7){\small{3}}
\put(8.7,0.6){\small{4}}
\put(4.7,-1.5){\small{5}}
\put(0,0.6){\small{6}}
\put(-1.2,4.7){\small{7}} 
\put(0,8.7){\small{8}}

\end{picture} \quad
\setlength{\unitlength}{2mm}
\begin{picture}(12,12)


\bigcircle{5}{5}{5}

\put(5,10){\circle*{0.3}}
\put(8.53,8.53){\circle*{0.3}}
\put(10,5){\circle*{0.3}}
\put(8.53,1.46){\circle*{0.3}}
\put(5,0){\circle*{0.3}}
\put(1.46,1.46){\circle*{0.3}}
\put(0,5){\circle*{0.3}}
\put(1.46,8.53){\circle*{0.3}}

\qbezier[500](0,5)(5,5)(1.46,1.46)
\qbezier[500](10,5)(5,5)(8.53,1.46)
\qbezier[500](1.46,8.53)(5,5)(1.46,1.46)
\qbezier[500](10,5)(5,5)(5,10)
\qbezier[500](8.53,1.46)(5,5)(5,10)
\qbezier[500](1.46,8.53)(5,5)(0,5)

\put(4.7,10.3){\small{1}}
\put(8.7,8.7){\small{2}} 
\put(10.3,4.7){\small{3}}
\put(8.7,0.6){\small{4}}
\put(4.7,-1.5){\small{5}}
\put(0,0.6){\small{6}}
\put(-1.2,4.7){\small{7}} 
\put(0,8.7){\small{8}}

\end{picture} \end{gathered} \right \}.$}
\end{center}
\caption{The balanced words coming from the boundary path matrices in Figure~\ref{ex:ppprob} and the corresponding noncrossing partitions.}
\label{ex:ppprowords}
\end{figure}

\begin{corollary}[D.\ Rush, X.\ Shi]
\label{cor:vicconj}
	Let $C_{m+n+1}$ act on $J([2]\times[m]\times[n])$ by $\row$.  Then $(J([2]\times[m]\times[n]), M_{2,n,m}(q), C_{m+n+1})$ exhibits the CSP.
\end{corollary}

\begin{proof}
	$J([2]\times[m]\times[n])$ under $\row$ is in equivariant bijection with noncrossing partitions of $[n+m+1]$ into $m+1$ blocks under rotation, which is known to exhibit the CSP~\cite{reiner2004cyclic}.
\end{proof}

We suspect that this bijection can be extended to $[\ell]\times[m]\times[n]$ for $\ell>2$; such a bijection would send an element of $J([\ell]\times[m]\times[n])$ to some noncrossing combinatorial object with $\ell+m+n-1$ external vertices, such that promotion translates to rotation of those vertices.  

\section{ASMs and TSCCPPs}
\label{sec:asmtsscpp}

We apply our methods to the alternating sign matrix and totally symmetric self-complementary plane partition posets, both of which are \rcposet{}s of height greater than one.  Given our previous results, it is natural to consider these posets as they both consist of layers of type $A$ positive root posets.  In particular, we give an equivariant bijection between ASMs under rowmotion and ASMs under B.\ Wieland's gyration, and we define two actions with related orders on ASMs and TSSCPPs.

\subsection{The ASM Poset}
\label{sec:asmp}

\begin{definition} \rm
\label{def:asm}
	An \emph{alternating sign matrix} (ASM) of order $n$ is an $n\times n$ matrix with entries $0$, $1$, or $-1$ whose rows and columns sum to 1 and whose nonzero entries in each row and column alternate in sign.
\end{definition}

Figure~\ref{ex:asm33} gives the $3\times 3$ ASMs.

\begin{figure}[ht]
\begin{center}
\scalebox{0.9}{
$
\left(
\begin{array}{rrr}
1 & 0 & 0 \\
0 & 1 & 0\\
0 & 0 & 1
\end{array} \right)
\left(
\begin{array}{rrr}
1 & 0 & 0 \\
0 & 0 & 1\\
0 & 1 & 0
\end{array} \right)
\left(
\begin{array}{rrr}
0 & 1 & 0 \\
1 & 0 & 0\\
0 & 0 & 1
\end{array} \right)
\left(
\begin{array}{rrr}
0 & 1 & 0 \\
0 & 0 & 1\\
1 & 0 & 0
\end{array} \right)
\left(
\begin{array}{rrr}
0 & 0 & 1 \\
1 & 0 & 0\\
0 & 1 & 0
\end{array} \right)
\left(
\begin{array}{rrr}
0 & 0 & 1 \\
0 & 1 & 0\\
1 & 0 & 0
\end{array} \right)
\left(
\begin{array}{rrr}
0 & 1 & 0 \\
1 & -1 & 1\\
0 & 1 & 0
\end{array} \right)
$}
\end{center}
\caption{The seven $3\times 3$ ASMs.}
\label{ex:asm33}
\end{figure}

ASMs have been objects of much study over the nearly three decades since W.\ Mills, D.\ Robbins, and H.\ Rumsey conjectured~\cite{MRRASMDPP} that the total number of $n \times n$ ASMs is

\begin{equation}
\label{eq:product}
\prod_{j=0}^{n-1} \frac{(3j+1)!}{(n+j)!}.
\end{equation}

This conjecture was proved 13 years later, independently---and by vastly different methods---by D.~Zeilberger~\cite{ZEILASM} and G.\ Kuperberg~\cite{KUP_ASM_CONJ}, and many new developments and directions have emerged since.  Nevertheless, an outstanding open problem is to find an explicit bijection between $n \times n$ ASMs and either of two sets of combinatorial objects known to be equinumerous with them: totally symmetric self-complementary plane partitions (TSSCPPs) inside a $2n\times 2n\times 2n$ box and descending plane partitions with largest part at most $n$.

We begin by recalling the poset interpretation of ASMs, first introduced by A.\ Lascoux and M.-P.\ Sch{\"u}tzenberger in~\cite{TREILLIS}. This poset is usually defined using monotone triangles, but we choose to define it equivalently using height functions because of their connection with gyration in Section~\ref{sec:gy}.  For many more interpretations of ASMs, see~\cite{propp2002many}.

\begin{definition} \rm
\label{def:heightfcn}
	A \emph{height function} of order $n$ is an $(n+1)\times (n+1)$ matrix $(h_{i,j})_{0\leq i,j \leq n}$ with $h_{0,k} = h_{k,0} = k$ and $h_{n,k} = h_{k,n} = n-k$ for $0 \leq k \leq n$, and such that adjacent entries in any row or column differ by $1$.
\end{definition}

The height functions of order $3$ are given in Figure~\ref{ex:heightfuncs}.

\begin{figure}[htb]
\begin{center}
$$\left(
\begin{array}{cccc}
 0 & 1 & 2 & 3 \\
 1 & 0 & 1 & 2 \\
 2 & 1 & 0 & 1 \\
 3 & 2 & 1 & 0
\end{array}
\right) \left(
\begin{array}{cccc}
 0 & 1 & 2 & 3 \\
 1 & 2 & 1 & 2 \\
 2 & 1 & 0 & 1 \\
 3 & 2 & 1 & 0
\end{array}
\right) \left(
\begin{array}{cccc}
 0 & 1 & 2 & 3 \\
 1 & 0 & 1 & 2 \\
 2 & 1 & 2 & 1 \\
 3 & 2 & 1 & 0
\end{array}
\right) 
 \left(
\begin{array}{cccc}
 0 & 1 & 2 & 3 \\
 1 & 2 & 1 & 2 \\
 2 & 3 & 2 & 1 \\
 3 & 2 & 1 & 0
\end{array}
\right)$$ 
$$\left(
\begin{array}{cccc}
 0 & 1 & 2 & 3 \\
 1 & 2 & 3 & 2 \\
 2 & 1 & 2 & 1 \\
 3 & 2 & 1 & 0
\end{array}
\right) \left(
\begin{array}{cccc}
 0 & 1 & 2 & 3 \\
 1 & 2 & 3 & 2 \\
 2 & 3 & 2 & 1 \\
 3 & 2 & 1 & 0
\end{array}
\right) 
\left(
\begin{array}{cccc}
 0 & 1 & 2 & 3 \\
 1 & 2 & 1 & 2 \\
 2 & 1 & 2 & 1 \\
 3 & 2 & 1 & 0
\end{array}
\right)$$
\end{center}
\caption{The seven height functions of order $3$.}
\label{ex:heightfuncs}
\end{figure}

\begin{proposition}[\cite{ELKP_DOMINO1}]
\label{prop:asmhf}
A bijection between $n\times n$ ASMs and height functions of order $n$ is given by mapping an ASM $(a_{ij})_{1\leq i,j \leq n}$ to the height function $\left(i+j-2\left(\sum_{i'=1}^i\sum_{j'=1}^j a_{i'j'}\right) \right)_{0\leq i,j\leq n}$.
\end{proposition}

Height functions of order $n$ have a partial ordering given by componentwise comparison of entries. This poset is a distributive lattice, and turns out to be the MacNeille completion of the Bruhat order on the symmetric group~\cite{TREILLIS}.  We denote the poset of join irreducibles as $\aposet_n$, so that $J(\aposet_n)$ is in bijection with the set of $n \times n$ ASMs. See~\cite{StrikerPoset} for further discussion.

For convenience---and in analogy with the construction of TSSCPPs in Section~\ref{sec:tsscpp}---we construct $\aposet_n$ directly as a layering of successively smaller root posets $\Phi^+(A_i)$.

\begin{definition} \rm
\label{def:asmposet}
	Define $\aposet_n$ to be the poset with the elements from $\Phi^+(A_1), \Phi^+(A_2) \ldots, \Phi^+(A_{n-1})$ and their covering relations, along with the additional covering relations that the root $e_j-e_k$ in $\Phi^+(A_i)$ covers the roots $e_j-e_k$ and $e_{j+1}-e_{k+1}$ in $\Phi^+(A_{i+1})$.
\end{definition}

\begin{proposition}
\label{prop:asmconstruction}
	$J(\aposet_n)$ is in bijection with the set of $n \times n$ ASMs.
\end{proposition}

\begin{figure}[htb]
\begin{center}
$\begin{array}{cc}
\begin{gathered}
\left(
\begin{array}{ccccc}
 0 & 1 & 2 & 3 & 4 \\
 1 & 2 & 1 & 2 & 3\\
 2 & 3 & 2 & 1 & 2\\
 3 & 2 & 1 & 0 & 1\\
4 & 3 & 2 & 1 & 0
\end{array}
\right) \end{gathered} &
\begin{gathered}\includegraphics[scale=0.5]{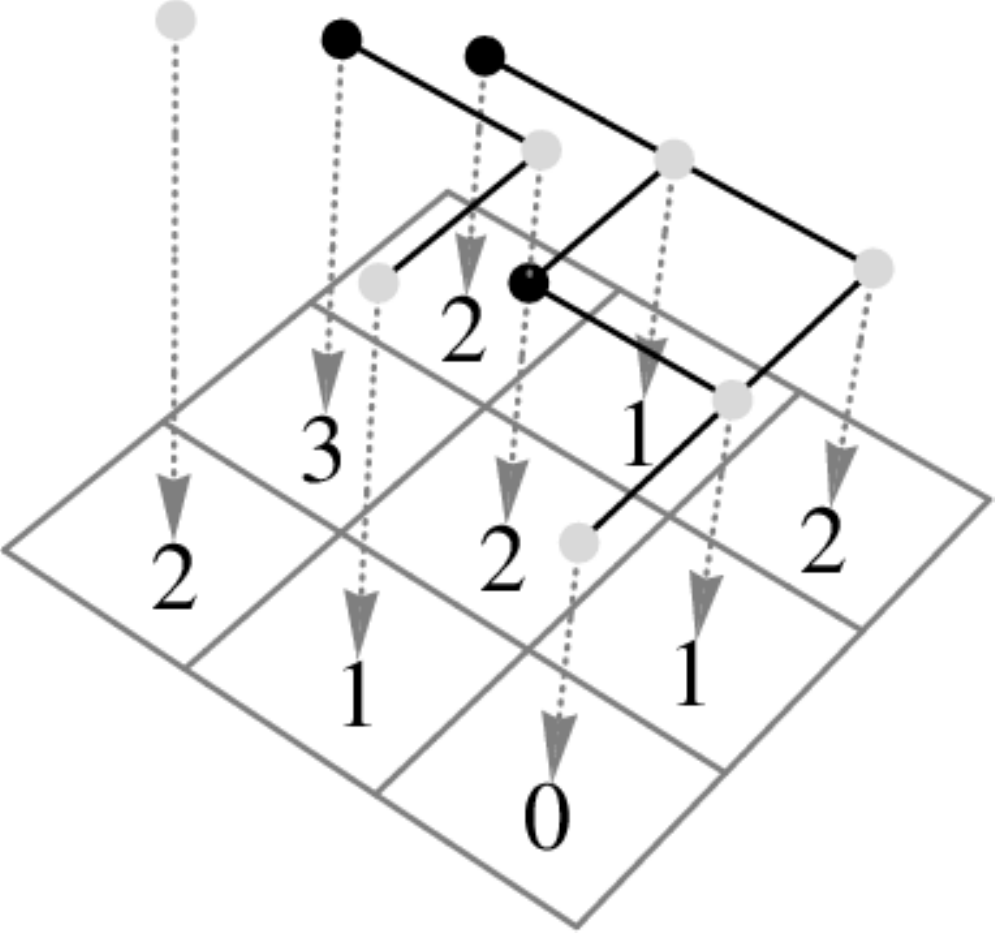}\end{gathered} \\
\end{array}$
\end{center}
\caption{On the left is a height function and on the right is the corresponding order ideal in $\aposet{}_4$.  The dotted arrows down indicate how the elements of the order ideal project down to the entries of the height function.}
\label{ex:posetA}
\end{figure}

See Figure~\ref{ex:posetA} for an example of the correspondence between height functions and order ideals.  Figure~\ref{ex:asm3} gives the order ideals of $\aposet_3$.  

\begin{figure}[htb]
\begin{center} $\left \{ \begin{gathered}
\scalebox{0.65}{
$\xymatrix @-1.2pc {
& \circ & & & \circ \\
\ar@{-}[ur] \ar@{-}[urrrr] \circ & & \ar@{-}[ul] \ar@{-}[urr] \circ & & }$
$\xymatrix @-1.2pc {
& \circ & & & \bullet \\
\ar@{-}[ur] \ar@{-}[urrrr] \bullet & & \ar@{-}[ul] \ar@{-}[urr] \bullet & & }$
$\xymatrix @-1.2pc {
& \bullet & & & \circ \\
\ar@{-}[ur] \ar@{-}[urrrr] \bullet & & \ar@{-}[ul] \ar@{-}[urr] \bullet & & }$
$\xymatrix @-1.2pc {
& \circ & & & \circ \\
\ar@{-}[ur] \ar@{-}[urrrr] \bullet & & \ar@{-}[ul] \ar@{-}[urr] \circ & & }$
$\xymatrix @-1.2pc {
& \circ & & & \circ \\
\ar@{-}[ur] \ar@{-}[urrrr] \circ & & \ar@{-}[ul] \ar@{-}[urr] \bullet & & }$
$\xymatrix @-1.2pc {
& \bullet & & & \bullet \\
\ar@{-}[ur] \ar@{-}[urrrr] \bullet & & \ar@{-}[ul] \ar@{-}[urr] \bullet & & }$
$\xymatrix @-1.2pc {
& \circ & & & \circ \\
\ar@{-}[ur] \ar@{-}[urrrr] \bullet & & \ar@{-}[ul] \ar@{-}[urr] \bullet & & }$} \end{gathered} \right \}.$
\end{center}
\caption{The seven order ideals in $J(\aposet_3)$.  They form a single orbit under superpromotion (Definition~\ref{def:super}).}
\label{ex:asm3}
\end{figure}
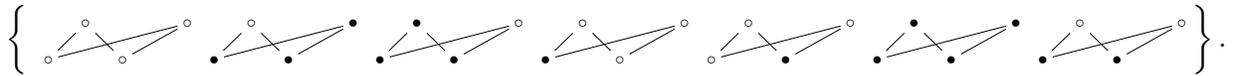

We can draw $\aposet_n$ as an \rcposet{} of height $h=n-1$ by sending $e_i-e_j$ in $\Phi^+(A_k)$ to the position $(i+j+n-k,i-j+n-k)$.  See Figure~\ref{ex:asm4}.

\begin{figure}[htb]
\begin{center}
$\begin{gathered}\includegraphics[scale=0.5]{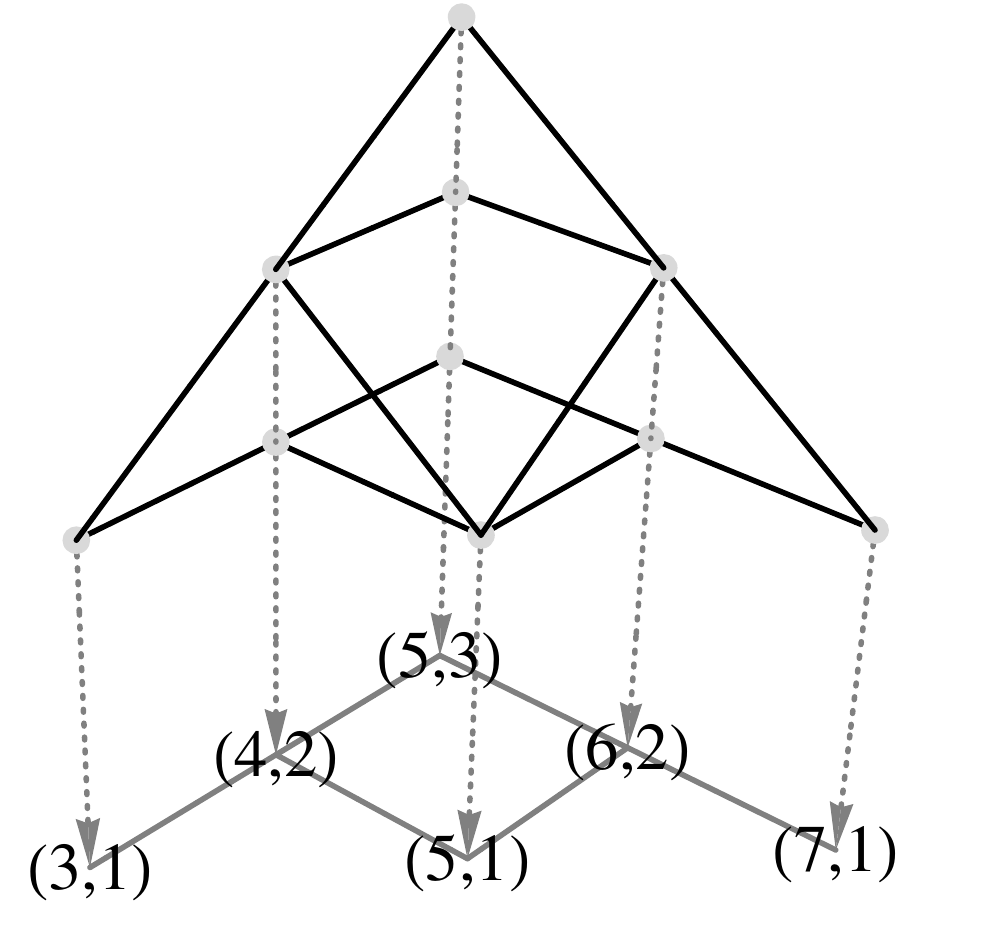}\end{gathered}$
\end{center}
\caption{$\aposet_4$ drawn as an \rcposet{} of height 3.  When there are multiple elements with the same position, subsequent elements are raised; the position is indicated by a dotted arrow down.  Covering relations are drawn with solid black lines, and are projected down as solid gray lines.}
\label{ex:asm4}
\end{figure}

Since $\aposet_n$ is an \rcposet{}, Theorem~\ref{thm:maintheorem} applies, giving an equivariant bijection between promotion and rowmotion on the ASM poset.

\begin{corollary}
	There is an equivariant bijection between $J(\aposet_n)$ under $\row$ and under $\pro$.
\end{corollary}

Interestingly, a conjugate to rowmotion and promotion in the toggle group of $\aposet_n$ has already been studied.

\subsection{Rowmotion and Gyration}
\label{sec:gy}

\begin{definition} \rm
\label{def:fpl}
	Consider the grid $[n] \times [n]$.  A \emph{fully-packed loop configuration (FPL)} of order $n$ is a set of paths that begin and end only at every second outward-pointing edge, such that each of the $n^2$ vertices within the grid lies on exactly one path.
\end{definition}

Figure~\ref{ex:fplgyration} gives the FPLs of order $3$.

\begin{figure}[ht]
\begin{center}
\includegraphics[scale=0.75]{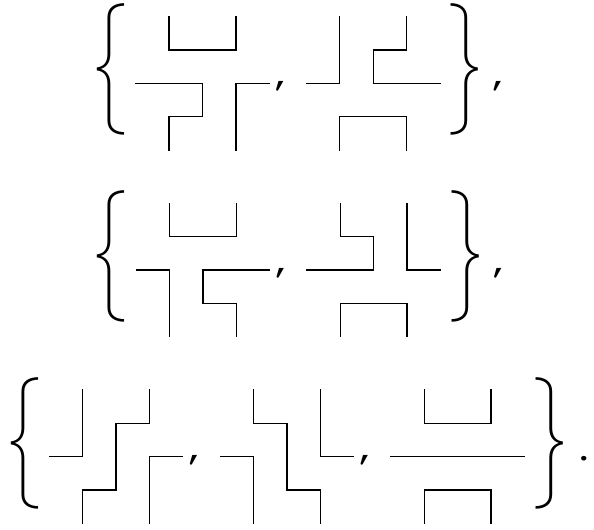}
\end{center}
\caption{The seven FPLs of order $3$.  They break into three orbits under gyration.}
\label{ex:fplgyration}
\end{figure}

We recall from  Proposition~\ref{prop:asmhf} that height functions of order $n$ are in bijection with $n\times n$ ASMs. Thus, FPLs can be seen to be in bijection with ASMs through height functions via the following.

\begin{proposition}[\cite{ELKP_DOMINO1}]
\label{prop:heightofpl}
	Height functions of order $n$ are in bijection with FPLs of order $n$.
\end{proposition}

\begin{proof}
	We sketch one direction of the bijection (see \cite{ELKP_DOMINO1}).

	Draw directed edges between adjacent entries in the height function matrix, pointing from the smaller value to the larger value.  Rotate each of these edges a quarter-turn counterclockwise about its midpoint and label each vertex even or odd according to the parity of the sum of its row and column indices.  Now delete all edges that exit odd vertices and enter even vertices, and unorient the remaining edges.
\end{proof}

\begin{definition} \rm
	Pairing up the boundary edges of each path reduces the FPL to a noncrossing matching on $2n$ vertices.  This matching is called the \emph{link pattern} of the FPL.
\end{definition}

In 2000, B.\ Wieland defined an action called \emph{gyration} on FPLs, which he proved rotated the corresponding link pattern~\cite{Wieland2000}.  This resolved the refinement by H.\ Cohn and J.\ Propp of a conjecture of C.\ Bosley and L.\ Fidkowski that the number of FPLs with a certain link pattern depends only on the link pattern up to rotation.  Similar actions had been studied by W.\ Mills, D.\ Robbins, and H.\ Rumsey in~\cite{MRRASMDPP}, though without the combinatorial significance of B.\ Wieland's result.

In 2010, L.\ Cantini and A.\ Sportiello generalized gyration in their proof of the Razumov-Stroganov conjecture that the number of FPLs with a given link pattern appears as the ground state components of the $O(1)$ loop model of statistical physics~\cite{RAZSTROGpf}.

\begin{definition} \rm
	Given an FPL, its \emph{gyration} is computed by first visiting all squares with lower left-hand corner $(i,j)$ for which $i+j$ is even, and then all squares for which $i+j$ is odd, swapping the edges around a square if the edges are parallel and otherwise leaving them fixed.
\end{definition}

Figure~\ref{ex:fplgyration} lists the FPLs by orbits under gyration.  We can define gyration directly on height functions.

\begin{proposition}
Gyration acts on height functions 
$(h_{i,j})_{0\leq i,j\leq n}$ 
by visiting all entries $h_{i,j}$, first those for which $i+j$ is even, and then those for which $i+j$ is odd,  
changing $h_{i,j}$ to its other possible value if each adjacent entry is equal and otherwise leaving it fixed.
\end{proposition}

Using this definition of gyration on height functions, we may interpret gyration directly in terms of the toggle group of the poset $\aposet_n$.

\begin{proposition}
\label{pro:gy135}
	Gyration acts as $\row_{135 \ldots 246 \ldots}$ on $J(\aposet_n)$.
\end{proposition}

\begin{proof}
	The interior height function entries on a diagonal with $i+j$ of fixed parity correspond to poset elements in rows of the opposite parity, so gyration moves through the poset toggling elements in odd rows first, then those in even rows.  This corresponds to $\row_{135 \ldots 246 \ldots}$ on $J(\aposet_n)$.
\end{proof}

Therefore, by Lemma \ref{lem:symconj}, we conclude that rowmotion and gyration are conjugate elements.

\begin{theorem}
	There is an equivariant bijection between $J(\aposet_n)$ under rowmotion and under gyration.
\end{theorem}

\subsection{ASM Superpromotion}
\label{sec:asmaction}

	Though gyration rotates FPL link patterns with order 2n, on the FPLs themselves, gyration has order greater than $2n$ for $n>4$. Though the order of gyration does not seem to adhere to a simple pattern, we can conclude that it is always divisible by $2n$, since for each $n$ there is a link pattern with order exactly $2n$.  For example, the order of gyration (and rowmotion/promotion) on $\aposet_n$ for some small values of $n$ is: $$\begin{array}{|c|c|c|c|c|c|c|c|} \hline
n & 1 & 2 & 3 & 4 & 5 & 6 & 7 \\ \hline
\text{Order of Gyration}  & 1 & 2 & 6 & 8 & 20 & 2520 & 3686760 \\ \hline
\end{array}.$$
	We define a new action on ASMs that has order divisible by $3n-2$.  We will compare this action with $\row$ on the TSSCPP poset, which by Theorem~\ref{thm:tsscpporder} also has order divisible by $3n-2$. 

\begin{definition} \rm
\label{def:super}
	Define \emph{ASM superpromotion} on $\aposet_n$ by $\spro = \prod_{i=1}^{n-1} \pro_i$ where each $\pro_i$ acts as promotion on $\Phi^+(A_i)$.
\end{definition}

Figure~\ref{ex:asm3} lists the single cycle of $J(\aposet_3)$ under
$\spro$.

\begin{theorem}
\label{thm:asmorder}
$J(\aposet_n)$ under $\spro$ has order divisible by $3n-2$.
\end{theorem}

\begin{proof}
	We show that $\spro^{3n-2}$ applied to the empty order ideal is the identity.

	Applying $\spro^{n-1}$ to the empty order ideal gives the order ideal such that the restriction to the $\Phi^+(A_{k})$ in $\aposet_n$ is the order ideal containing the elements $e_i-e_j$, for $i<j\leq 2k-n+1$.  Applying $\spro^{n}$ to this order ideal then gives the order ideal containing all elements of $\aposet_n$.  Finally, $\spro^{n-1}$ takes us back to the empty order ideal by removing one row from the order ideal at a time.
\end{proof}

While this action is of order $3n-2$ for $n \leq 6$, it has order $3 \cdot (3 \cdot 7-2) = 57$ for $n=7$ (see Figure~\ref{ex:asmpptable}).

The obvious $q$-analogue of~(\ref{eq:product}) is $\prod_{j=0}^{n-1} \frac{(3j+1)!_q}{(n+j)!_q},$ which is the generating function of descending plane partitions (DPPs) with largest part at most $n$~\cite{MRRASMDPP}.  In general, a corresponding statistic on ASMs or TSSCPPs is not known.  For permutation matrices, however, such a statistic is known, and there is a statistic-preserving bijection to a subclass of DPPs~\cite{STRIKERDPP}.  It is not hard to check that this polynomial cannot exhibit the CSP with superpromotion, though the first time it fails is at $n=6$.  CSPs for ASMs acted on by cyclic groups of small order have been shown for quarter-turns and half-turns in~\cite{REUpaper}.  For a related result on vertically symmetric ASMs, see~\cite{Williams2008}.

Much as FPLs demonstrate that the order of gyration is divisible by $2n$, our hope is that there exists a generalization of the map stated in Theorem~\ref{prop:proback} to take ASMs to some noncrossing combinatorial object on $3n-2$ external vertices.

\subsection{The TSSCPP Poset}
\label{sec:tsscpp}

For our purposes, we need only define the poset whose order ideals are in bijection with totally symmetric self-complementary plane partitions; see~\cite{StrikerPoset} for a definition of TSSCPPs and an explanation of how the partial order is obtained (this partial order is the same as the partial order on the \emph{magog} triangles of~\cite{ZEILASM}).

In analogy with the construction of the ASM poset in Section~\ref{sec:asmp}, we construct the TSSCPP poset $\tposet_n$ as a layering of successive root posets $\Phi^+(A_i)$. 

\begin{definition} \rm
\label{def:tsscppposet}
	Define $\tposet_n$ to be the poset with elements from $\Phi^+(A_1),\Phi^+(A_2), \ldots, \Phi^+(A_{n-1})$ and their covering relations, along with the additional covering relations that the root $e_j-e_k$ in $\Phi^+(A_i)$ is covered by the root $e_j-e_k$ in $\Phi^+(A_{i+1})$.
\end{definition}

\begin{proposition}
\label{prop:tsscppconstruction}
	$J(\tposet_n)$ is in bijection with the set of TSSCPPs inside a $2n\times 2n\times 2n$ box.
\end{proposition}

Figure~\ref{ex:tsscpp3} gives the order ideals of $\tposet_3$.

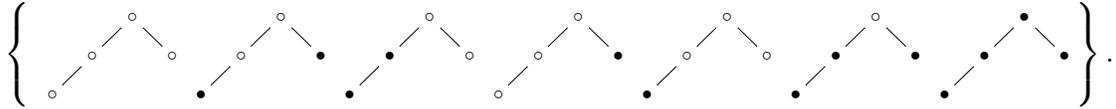
\begin{figure}[ht]
\begin{center}
$\left \{ \begin{gathered} \scalebox{0.7}{
$\xymatrix @-1.2pc {
& & \circ & \\
&\ar@{-}[ur] \circ & & \ar@{-}[ul] \circ \\
 \ar@{-}[ur] \circ & & &\\
}$
$\xymatrix @-1.2pc {
& & \circ & \\
&\ar@{-}[ur] \circ & & \ar@{-}[ul] \bullet \\
 \ar@{-}[ur] \bullet & & &\\
}$
$\xymatrix @-1.2pc {
& & \circ & \\
&\ar@{-}[ur] \bullet & & \ar@{-}[ul] \circ \\
\ar@{-}[ur] \bullet & & &\\
}$
$\xymatrix @-1.2pc {
& & \circ & \\
&\ar@{-}[ur] \circ & & \ar@{-}[ul] \bullet \\
\ar@{-}[ur] \circ & & & \\
}$
$\xymatrix @-1.2pc {
& & \circ & \\
&\ar@{-}[ur] \circ & & \ar@{-}[ul] \circ \\
 \ar@{-}[ur] \bullet & & & \\
}$
$\xymatrix @-1.2pc {
& & \circ & \\
& \ar@{-}[ur] \bullet & & \ar@{-}[ul] \bullet \\
 \ar@{-}[ur] \bullet & & & \\
}$
$\xymatrix @-1.2pc {
& & \bullet & \\
& \ar@{-}[ur] \bullet & & \ar@{-}[ul] \bullet \\
 \ar@{-}[ur] \bullet & & & \\
}$}\end{gathered} \right \}.$
\end{center}
\caption{There are seven order ideals in $J(\tposet_3)$.  They form a single orbit under rowmotion.}
\label{ex:tsscpp3}
\end{figure}

We can draw $\tposet_n$ as an \rcposet{} of height $h= \lfloor \frac{n}{2} \rfloor$ by sending $e_i-e_j$ in $\Phi^+(A_k)$ to the position $(i+j+k,j-i+k)$.  See Figure~\ref{ex:tsscpp4}.

\begin{figure}[ht]

\begin{center}
$\begin{gathered}\includegraphics[scale=0.5]{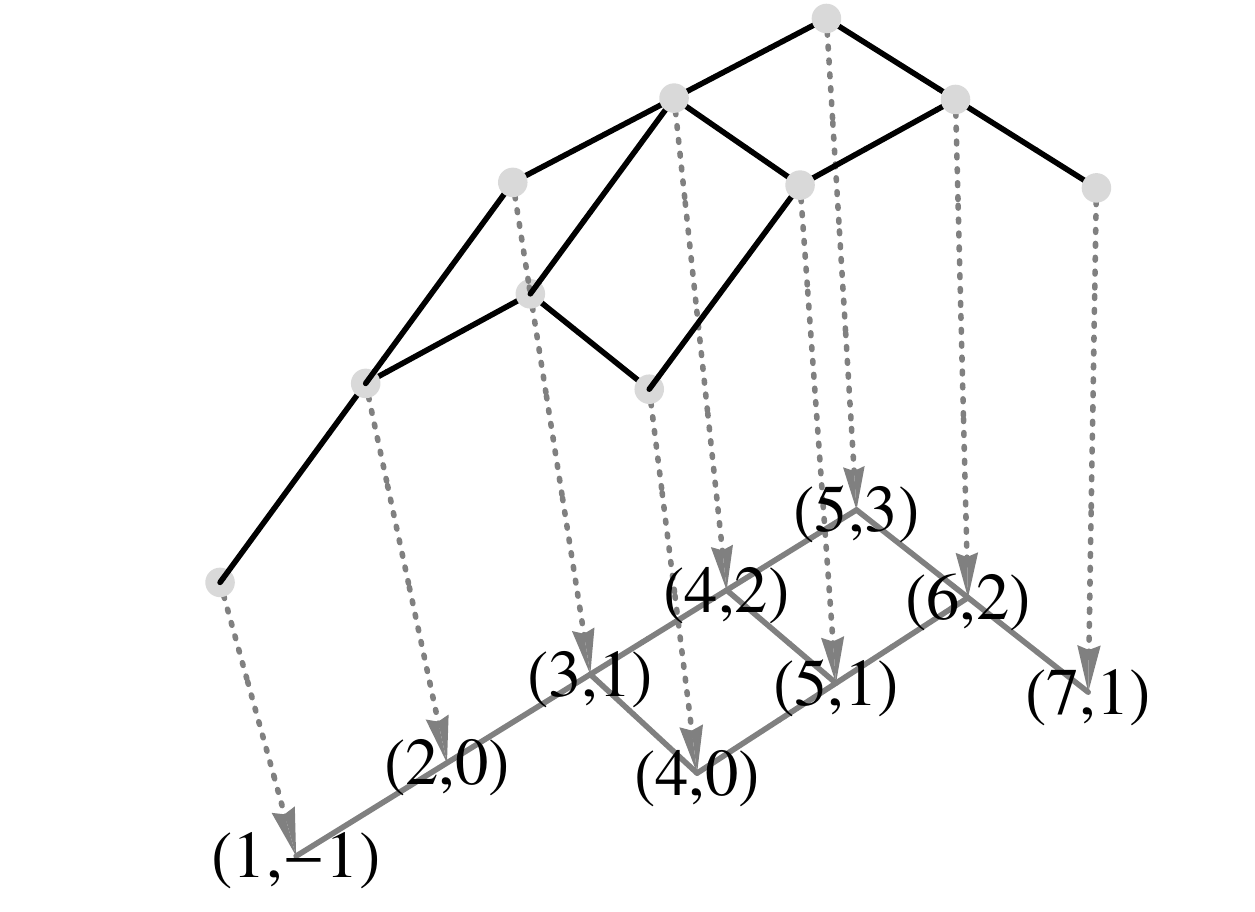}\end{gathered}$
\end{center}
\caption{$\tposet_4$ drawn as an \rcposet{} of height 2.  When there are two elements with the same position, the second element is raised; the position is indicated by a dotted arrow down.  Covering relations are drawn with solid black lines, and are projected down as solid gray lines.}
\label{ex:tsscpp4}
\end{figure}

Appealing once again to Theorem~\ref{thm:maintheorem}, we obtain the conjugacy of $\row$ and $\pro$.

\begin{corollary}
	There is an equivariant bijection between $J(\tposet_n)$ under $\row$ and under $\pro$.
\end{corollary}

\begin{theorem}
\label{thm:tsscpporder}
	$J(\tposet_n)$ under $\row$ has order divisible by $3n-2$.
\end{theorem}

\begin{proof}
	We show that $\row^{3n-2}$ applied to the empty order ideal is the identity.

	Applying $\row$ to the empty order ideal gives the full order ideal.  Applying $\row^{n-1}$ to this gives the order ideal such that the restriction to the $\Phi^+(A_{k})$ in $\tposet_n$ is the order ideal containing the elements $e_i-e_j$, for $i<j$ and $j-i\leq n-1-k$.  Applying $\row^{n-1}$ again returns this same order ideal, with the additional elements $e_i-e_j$ with $i$ even and $j-i=n-k$ in $\Phi^+(A_{k})$ (for $k>\frac{n}{2}$).  Finally, $\row^{n-1}$ takes us back to the empty order ideal.
\end{proof}

In analogy with FPLs and ASMs, we again expect a bijection from TSSCPPs to a noncrossing combinatorial object with $3n-2$ external vertices, such that promotion corresponds to rotation of those vertices.

Since the order of $\pro$ or $\row$ on $J(\tposet_n)$ and the order of $\spro$ on $J(\aposet_n)$ are related, one could hope to use the method of~\cite{Armstrong2011} to define a bijection from ASMs to TSSCPPs.  This method would inductively associate elements of $J(\aposet_n)$ that are naturally elements of $J(\aposet_{n-1})$, and then use the orbit structure to extend the bijection.  Unlike the situation for positive root posets~\cite{Armstrong2011}, however, it is not the case that every orbit contains such an element.  This method therefore fails---though for small $n$ it can be used to gather data when trying to find a bijection. See Figure~\ref{ex:asmpptable} for data on the orbit sizes of these actions.

\begin{figure}[ht]
\begin{center}
\begin{tabular}{|l||l|l||l|l|}
\hline
 & \multicolumn{2}{|c||}{$J(\aposet_n)$ under $\spro$} &
\multicolumn{2}{|c|}{$J(\tposet_n)$ under $\pro$ or $\row$}  \\ \hline
 \cline{1-5}
 & Orbit Size & Number of Orbits & Orbit Size & Number of Orbits \\ \hline
\multirow{1}{*}{$n=1$}
 & 1 & 1 & 1 & 1\\ \hline
\multirow{1}{*}{$n=2$}
 & 2 & 1 & 2 & 1 \\ \hline
\multirow{1}{*}{$n=3$}
 & 7 & 1 & 7 & 1 \\ \hline
\multirow{3}{*}{$n=4$}
 & 10 & 3 & 10 & 3 \\
 & 5 & 2 & 5 & 2 \\
 & 2 & 1 & 2 & 1 \\ \hline
\multirow{3}{*}{$n=5$}
 &  &  & 39 & 1 \\
 &  & & 26 & 1 \\
 & 13 & 33 & 13 & 28\\ \hline
\multirow{8}{*}{$n=6$}
 & & & 112 & 1\\
 & & & 96 & 2\\
 & & & 80 & 2\\
 & & & 64 & 5\\
 & & & 48 & 23\\
 & & & 32 & 30\\
 & & & 24 & 2\\
 & 16 & 456 & 16 & 277 \\
 & 8 & 16 & 8 & 13\\
 & 4 & 2 &  &  \\
 & 2 & 2 & 2 & 2 \\ \hline
\multirow{2}{*}{$n=7$}
 & 57 & 55 &  &  \\
 & 19 & 11327 & *  & * \\ \hline
\end{tabular}
\end{center}
\caption{Data for the orbits of $J(\aposet_n)$ under superpromotion and the orbits of $J(\tposet_n)$ under promotion or rowmotion. The $n=7$ TSSCPP data is omitted for space considerations.}
\label{ex:asmpptable}
\end{figure}

\section{Acknowledgments}

The authors thank Vic Reiner, Dennis Stanton, and Dennis White for many helpful conversations and suggestions, and Peter Webb for his insights regarding toggle groups.  They are particularly grateful to David B Rush and Xiaolin Shi for their work on plane partitions, to Vic Reiner for introducing them to rowmotion, and to the anonymous reviewers for their many helpful comments.

\bibliographystyle{amsplain}
\bibliography{proequalsrow2}

\end{document}